\documentclass[letterpaper,final,twocolumn]{IEEEtran}  

\usepackage{graphicx}
\usepackage{subfigure}
\usepackage{amsmath}
\usepackage{amsthm}
\usepackage{amssymb}
\usepackage{color}
\usepackage{xspace}

\usepackage{savesym}
\savesymbol{AND}

\usepackage{algorithmic}
\usepackage{algorithm}

\newcommand{\clusteralgo}{\texttt{\small clustering into
    bubbles}\xspace}%
\newcommand{\veloptalgo}{\texttt{\small bubble velocity
    optimization}\xspace}%
\newcommand{\velboundalgo}{\texttt{\small bounding optimal bubble
    velocity}\xspace}%
\newcommand{\schedalgo}{\texttt{\small schedule optimization}\xspace}%
\newcommand{\localvehcontrol}{\texttt{\small local vehicular}\xspace}%

\usepackage[colorlinks = true, linkcolor = blue, citecolor =
blue]{hyperref}

\newcommand{\mycomment}[1]{\hfill \COMMENT{\texttt{#1}}}

\newtheorem{theorem}{Theorem}[section]
\newtheorem{proposition}[theorem]{Proposition}
\newtheorem{lemma}[theorem]{Lemma}
\newtheorem{corollary}[theorem]{Corollary}
\newtheorem{remark}[theorem]{Remark}
\newtheorem{definition}[theorem]{Definition}

\newcommand{\longthmtitle}[1]{\mbox{}\textit{(#1).}}

\newcommand{\real}{\ensuremath{\mathbb{R}}}
\newcommand{\realpositive}{\ensuremath{\mathbb{R}_{>0}}}
\newcommand{\realnonnegative}{\ensuremath{\mathbb{R}_{\ge 0}}}
\newcommand{\integers}{{\mathbb{Z}}}
\newcommand{\integerspositive}{{\mathbb{N}}}
\newcommand{\integersnonnegative}{{\mathbb{N}}_0}

\newcommand{\union}{\ensuremath{\operatorname{\cup}}}

\newcommand{\Cc}{\mathcal{C}}
\newcommand{\Dc}{\mathcal{D}}

\newcommand{\Fc}{\mathcal{F}}

\newcommand{\Ic}{\mathcal{I}}
\newcommand{\Lc}{\mathcal{L}}
\newcommand{\Mc}{\mathcal{M}}
\newcommand{\Nc}{\mathcal{N}}
\newcommand{\Pc}{\mathcal{P}}
\newcommand{\Qc}{\mathcal{Q}}
\newcommand{\Rc}{\mathcal{R}}
\newcommand{\Sc}{\mathcal{S}}

\newcommand{\Zc}{\mathcal{Z}}

\newcommand{\tauocc}{\tau^{\text{occ}}}
\newcommand{\avgvel}{\bar{v}}

\newcommand{\lowvel}{\bar{v}^m}
\newcommand{\hivel}{\bar{v}^M}

\newcommand{\xv}{x^v}
\newcommand{\vv}{v^v}
\newcommand{\dxv}{\dot{x}^v}
\newcommand{\dvv}{\dot{v}^v}
\newcommand{\uv}{u^v}
\newcommand{\nunom}{\nu^{\text{nom}}}
\newcommand{\Dcnom}{\Dc^{\text{nom}}}
\newcommand{\Tnom}{T^{\text{nom}}}
\newcommand{\Tfol}{T^{\text{fol}}}
\newcommand{\taue}{\tau^e}
\newcommand{\taul}{\tau^l}
\newcommand{\taumin}{\tau^{\text{min}}}
\newcommand{\Ta}{T^a}
\newcommand{\Tx}{T^{\text{exit}}}
\newcommand{\ulinevv}{\underline{v}^v}
\newcommand{\Tiat}{T^{iat}}

\newcommand{\maxNbubbles}{\bar{\Nc}}
\newcommand{\Ncbar}{\bar{\Nc}}
\newcommand{\stagingset}{\Zc^s}
\newcommand{\midset}{\Zc^m}
\newcommand{\exitset}{\Zc^e}
\newcommand{\Nnewvehicles}{n^{ua}}   
\newcommand{\timestepCS}{T_{cs}}

\newcommand{\guc}{g_{uc}}
\newcommand{\gsf}{g_{sf}}
\newcommand{\gus}{g_{us}}

\newcommand{\hilim}[1]{H^#1}
\newcommand{\sigP}[1]{\sigma_#1}
\newcommand{\card}[1]{|#1|}
\newcommand{\Len}[1]{\card{#1}}
\newcommand{\first}[1]{#1(1)}
\newcommand{\last}[1]{#1(\Len{#1})}
\newcommand{\until}[1]{\{1,\dots,{#1}\}}
\newcommand{\branches}{\{1,2,3,4\}}

\newcommand{\bartauocc}{\bar{\tau}^{\text{occ}}}

\newcommand{\satu}[1]{\left[ #1 \right]_{u_m}^{u_M}}

\newcommand{\oprocendsymbol}{\hbox{$\bullet$}}
\newcommand{\oprocend}{\relax\ifmmode\else\unskip\hfill\fi\oprocendsymbol}

\graphicspath{{figures/}}

\begin{document}

\title{Hierarchical-distributed optimized coordination
  \\
  of intersection traffic\thanks{A preliminary version of this work
    appeared as~\cite{PT-JC:15-necsys} at the 5th IFAC Workshop on
    Distributed Estimation and Control in Networked Systems.}}




\author{Pavankumar Tallapragada \qquad Jorge Cort{\'e}s
  \thanks{Pavankumar Tallapragada and Jorge Cort{\'e}s are with the
    Department of Mechanical and Aerospace Engineering, University of
    California, San Diego {\tt\small
      \{ptallapragada,cortes\}@ucsd.edu}}%
}

\maketitle

\begin{abstract}
  This paper considers the problem of coordinating the vehicular
  traffic at an intersection and on the branches leading to it for
  minimizing a combination of total travel time and energy
  consumption. We propose a provably safe hierarchical-distributed
  solution to balance computational complexity and optimality of the
  system operation. In our design, a central intersection manager
  communicates with vehicles heading towards the intersection, groups
  them into clusters (termed bubbles) as they appear, and determines
  an optimal schedule of passage through the intersection for each
  bubble. The vehicles in each bubble receive their schedule and
  implement local distributed control to ensure system-wide
  inter-vehicular safety while respecting speed and acceleration
  limits, conforming to the assigned schedule, and seeking to optimize
  their individual trajectories. Our analysis rigorously establishes
  that the different aspects of the hierarchical design operate in
  concert and that the safety guarantees provided by the proposed
  design are satisfied. We illustrate its execution in a suite of
  simulations and compare its performance to traditional signal-based
  coordination over a wide range of system parameters.
\end{abstract}


\section{Introduction}

With rapidly growing urbanization and mobility needs of people across
the world, existing transportation systems are in critical need of
transformation. Apart from increased travel times, current burdened
transportation systems have the side effects of increased pollution,
increased energy consumption, and degradation of people's health, all
of which have an immeasurable cost on society. 
The complexity of the challenge requires a multi-pronged approach, one
of which is the development of new technologies. Emerging technologies
such as vehicle-to-vehicle (V2V) and vehicle-to-infrastructure (V2I)
communication, and computer-controlled vehicles offer the opportunity
to radically redesign our transportation systems, eliminating road
accidents and traffic collisions and positively impacting safety,
traveling ease, travel time, and energy consumption.

A particularly useful application of these technologies is the
coordination of traffic at and near intersections for a smoother (with
reduced stop-and-go) and fuel-efficient traffic flow. An intersection
manager with knowledge of the state of the traffic could schedule the
intersection crossings of the vehicles. With the assigned schedule,
individual vehicles could further optimize their travel to the
intersection in a fuel-efficient way.  In contrast to traditional
intersection management, networked vehicle technologies allow us to
coordinate the traffic not just \emph{within the intersection}, but
also by controlling the vehicles' behavior much before they arrive at
the intersection. Such a paradigm offers the possibility of
significantly reduced stop times and increased fuel efficiency, and is
the subject of this paper.

\subsubsection*{Literature review}
Much of the literature in the area of coordination-based intersection
management focuses on collision avoidance of vehicles \emph{within the
  intersection}. Supervisory intersection management (intervention
only when required to maintain safety by avoiding collisions) is
explored using discrete event abstractions
in~\cite{ED-AC-DDV-SL:13acc
} and reachable set computations in~\cite{AC-DDV:14,
  MRH-DC-LC-DDV:13}. The works~\cite{KD-PS:08,DF-TA-STW-PS-DY:11} and
references therein describe a multiagent simulation approach in which,
upon a reservation request from a vehicle, an intersection manager
accepts or rejects the reservation based on a simulation. Each vehicle
attempts to conform to its assigned reservation and, if this is
predicted not to be possible at any time, the reservation is
canceled.~\cite{HK-DC-PRK:11} also uses a reservation-based system to
schedule intersection crossing times and provides provably safe
maneuvers for vehicle following in a lane as well as for crossing the
intersection. \cite{RH-GRC-PF-HW:15, GRC-PF-HW-RH-JS:14} use a method
based on model predictive control to coordinate the intersection
crossing by vehicles and obtain suboptimal solutions to a linear
quadratic optimal control problem. \cite{MASK-etal:15} also proposes a
model predictive control approach in which collision-free intersection
crossing by vehicles is achieved through a combination of hard
no-collision constraints as well as a soft constraint in the form of a
term measuring collision risk in the cost
function. In~\cite{XQ-JG-FM-ADLF:14}, a heuristic policy assigns
priorities to the vehicles, while each vehicle applies a
priority-preserving control and legacy vehicles platoon behind a
computer-controlled car.  In this context, we note that the ability to
efficiently coordinate diminishes as the vehicles get closer to the
intersection. This is why here we take an expanded view of
intersection management that looks at the coordinated control of the
vehicles much before they arrive at the intersection.
The methods above are not suited for this setup or would prove to be
too computationally costly in such scenarios. An example of the
expanded view of intersection management is~\cite{DM-SK:14}, in which
a polling-systems approach is adopted to assign schedules, and then
optimal trajectories for all vehicles are computed sequentially in
order. Such optimal trajectory computations are costly and depend on
other vehicles' detailed plans, and hence the system is not robust.
Closer to this paper, the works \cite{QJ-GW-KB-MB:12,QJ-GW-KB-MB:13}
describe a hierarchical setup, with a central coordinator verifying
and assigning reservations, and with vehicles planning their
trajectories locally to platoon and to meet the assigned schedule.
The proposed solution is based on multiagent simulations and a
reservation-based scheduling (with the evolution of the vehicles
possibly forcing revisions to the schedule), both important
differences with respect to our approach. \cite{LL-DW-DY:14} is a
recent survey of traffic control with vehicular networks and provides
other related references.

\subsubsection*{Statement of contributions}

We propose a provably safe intersection management system aimed at
optimizing a combination of cumulative travel time and fuel usage for
all the vehicles.  Our first contribution is the idea of coordinating
the traffic at the intersection and on the branches leading to it in a
unified, holistic way. The basic observation is that planning and
controlling the vehicles from much before they arrive at the
intersection should lead to better overall coordination and
efficiency.

Our second set of contributions is a multi-layered design that
combines hierarchical and distributed control and is applicable to a
wide range of traffic conditions. Our hierarchical-distributed
approach offers a good balance between computational complexity of the
solution and optimal operation. The proposed system is composed of
three main aspects: (i) clustering to identify vehicles that platoon
before arriving at the intersection. We refer to such clusters of
vehicles as \emph{bubbles}.  We use the term bubble, rather than
platoon, to emphasize the dynamic, time-varying nature of the
cohesiveness of the group of vehicles as they travel towards the
intersection. With this terminology, the bubble becomes a rigid and
cohesive group (i.e., a platoon) by the time they cross the
intersection; (ii) a branch-and-bound scheduling algorithm that, using
aggregate information about the bubbles, allows a central intersection
manager to find the optimal schedule of bubble passage; and (iii) a
distributed control algorithm for the vehicles at the local level.
This control policy ensures that the vehicles of each bubble platoon
into a cohesive group when they cross the intersection and that each
bubble conforms to the schedule prescribed by the intersection
manager, while guaranteeing system-wide safety subject to speed limits
and acceleration saturation. Additionally, each vehicle seeks to
optimally control its trajectory whenever safety is not immediately
threatened.  The first two aspects are contributions of this paper
while the aspect of local vehicular control incorporates the
algorithmic solution from our previous work~\cite{PT-JC:17-tcns}.

Our third and final contribution is the technical analysis leading to
the provable safety of our design. In contrast to computationally
intensive multiagent simulation-based methods, we provide analytical
guarantees on correctness, safety, and performance. Further, the
results provide good intuition and fundamental and reliable principles
for future designs.  We do acknowledge that the development of
analytical guarantees comes at the cost of some conservatism in the
design. We have performed a suite of simulations comparing our
approach to traditional signal-based coordination that show a
significant improvement in the cumulative energy consumption for a
wide range of traffic densities and a more socially equitable
distribution of cost. However, the throughput of the intersection is
significantly less in our approach than that of signal-based
coordination except for low densities of traffic.  As a final note for
the reader's sake, we have made every effort in the presentation to
make the components of the paper understandable even if the proofs of
the technical results are skipped in a first reading.

\section{Preliminaries}

We present here some basic notation and concepts on graph theory used
throughout the paper.

\subsubsection*{Notation}
We let $\real$, $\realnonnegative$, $\integers$, $\integerspositive$,
and $\integersnonnegative$ denote the set of real, nonnegative real,
integer, positive integer, and nonnegative integer numbers,
respectively. For a non-empty ordered list $\Sc = \{j_1, \ldots,
j_s\}$, we let $\Len{\Sc}$ denote the cardinality of $\Sc$. Further,
$\Sc(i)$ denotes the $i^{\text{th}}$ element $j_i$ of $\Sc$. Thus,
$\last{\Sc}$ denotes the last element of~$\Sc$. For convenience, we
also use the notation $j \in \Sc$ ($j \notin \Sc$) to denote that $j$
is (is not) an element of~$\Sc$. For two ordered lists $\Sc_1$ and
$\Sc_2$, we let $\Sc_1 \setminus \Sc_2$ denote the ordered list of
elements that belong to $\Sc_1$ but not to $\Sc_2$, while preserving
the same order of~$\Sc_1$.
Given $u_m \le u_M$, $\satu{u}$ denotes the number $u$
lower and upper saturated by $u_m$ and $u_M$ respectively, i.e.,
\begin{equation*}
  \satu{u} \triangleq
  \begin{cases}
    u_m, \quad \text{if } u \leq u_m,
    \\
    u, \quad \text{if } u \in [u_m, u_M],
    \\
    u_M, \quad \text{if } u \geq u_M .
  \end{cases}
\end{equation*}

\subsubsection*{Graph theory}
We review basic notions following the exposition
in~\cite{RD:00,FB-JC-SM:08cor}.  A digraph of order $n$ is a pair $G =
(V,E)$, where $V$ is a set with $n$ elements called nodes and $E$ is a
set of ordered pair of nodes called edges.  A directed path is an
ordered sequence of nodes such that any ordered pair of nodes
appearing consecutively is an edge. A cycle is a directed path that
starts and ends at the same node and contains no repeated node except
for the initial and the final one. A digraph is acyclic if it has no
cycles.  A directed (or rooted) tree is an acyclic digraph with a
node, called root, such that any other node can be reached by one and
only one directed path starting at the root. If $(i,j)$ is an edge of
a tree, $i$ is the parent of $j$, and $j$ is the child of~$i$. A
node~$j$ is called a descendant of a node~$i$ if there is a directed
path from~$i$ to~$j$. Given a tree, a subtree rooted at $i$ is the
tree that has $i$ as its root and is composed by all its descendants
in the original tree.

\section{Problem statement}\label{sec:problem-statement}

Consider an intersection of length $\Delta$ with four incoming traffic
branches labeled by $\branches$, cf. Figure~\ref{fig:intersect}.
\begin{figure}[!htpb]
  \centering
  \includegraphics[width=0.6\linewidth]{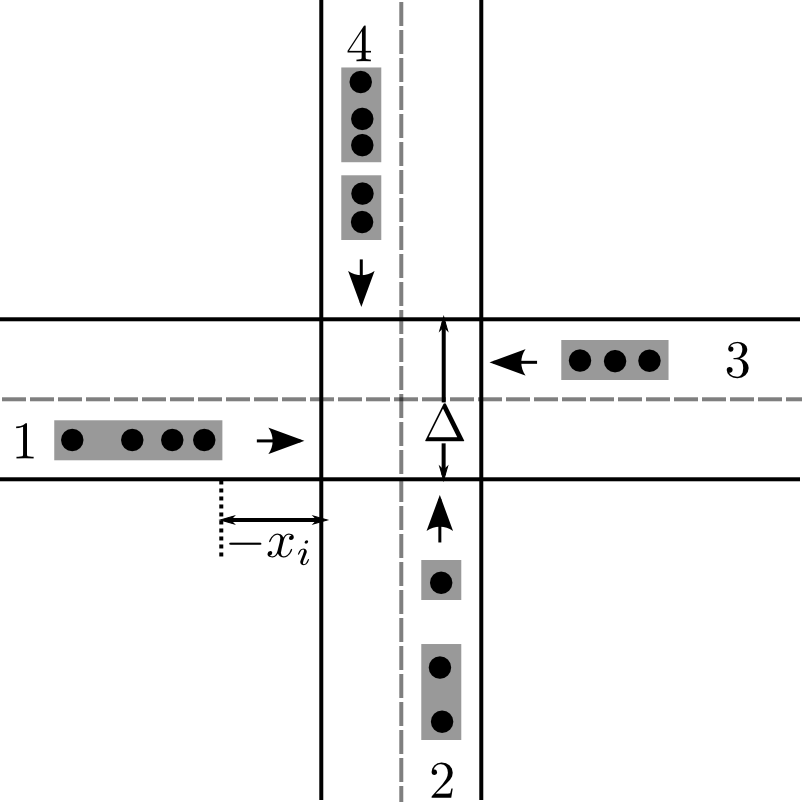}
  \caption{Traffic near an intersection. Black dots represent
    individual vehicles, which are clustered and contained within
    \emph{bubbles}, represented by grey boxes. $\Delta$ is the length
    of the intersection and the numbers $\branches$ are labels for
    the incoming branches.}\label{fig:intersect}
\end{figure}
For simplicity, we assume that (i) there is a single lane in each
direction, (ii) all vehicles are identical with length $L$, (iii)
vehicles do not turn at the intersection, (iv) the intersection at any
time may be used by vehicles from a single branch (v) there are no
sources or sinks for vehicles along the branches - all new traffic
appears at the beginning of the branches and must cross the
intersection. We discuss later in Remark~\ref{re:assumptions} the
extent to which these assumptions can be relaxed in our algorithmic
solution.  The dynamics of vehicle $j$ is a fully actuated
second-order system,
\begin{subequations}\label{eq:vehicle-dyn}
  \begin{align}
    \dxv_j(t) &= \vv_j(t) ,
    \\
    \dvv_j(t) &= \uv_j(t) ,
  \end{align}
\end{subequations}
where $\xv_j$, $\vv_j \in \real$ are the position (negative of the
distance from the front of the vehicle to the beginning of the
intersection) and velocity of the vehicle, respectively and $\uv_j(t)
\in [u_m, u_M]$, with $u_m\le 0 \le u_M$, is the input
acceleration. We use the superscript $v$ to emphasize that the state
and control variables refer to individual vehicles. We assume that
each branch has a maximum speed limit that the vehicles must
respect. For the sake of easing the notation, we assume that the speed
limit on all branches is the same and equals $v^M$. Thus, for each
vehicle $j$, $\vv_j(t)$ must belong to the interval $[0, v^M]$ for all
time $t$ that the vehicle is in the system.

Each vehicle is equipped with vehicle-to-vehicle (V2V) and
vehicle-to-infrastructure (V2I) communication capabilities. With V2I
communication, the vehicles inform a central \emph{intersection
  manager} (IM) about their positions and velocities and receive from
it commands such as prescribed time of arrival at the intersection. We
assume the IM has the necessary communication and computing
capabilities.  We seek a design solution that aims to minimize a cost
function~$C$ that models a combination of cumulative travel time and
cumulative fuel cost of the form
\begin{equation}\label{eq:cost-fun}
  C \triangleq \sum_{j} \int_{t^{\text{spawn}}_j}^{\Tx_j} ( W_T + | \uv_j | )
  \mathrm{d}t ,
\end{equation}
where $j$ is the vehicle index, $t^{\text{spawn}}_j$ is the time at
which vehicle~$j$ `spawns' into the problem domain and $\Tx_j$ is the
time at which the vehicle exits the intersection, i.e., $\xv_j(\Tx_j)
= \Delta + L$. The weight $W_T$ sets the relative importance of travel
time versus fuel cost. The vehicles over which the cost is summed may
be chosen in different ways - for example it may be over all vehicles
that cross the intersection in a time period or it may be over a fixed
number of vehicles. The constraints in the problem arise from the
speed limit, bounds on vehicle acceleration and deceleration, and the
safety requirements - which require scheduling the intersection
crossing of the vehicles and maintenance of safe distance between the
vehicles. Solving this problem at the level of individual vehicles is
computationally expensive and not scalable. Thus, we aim to synthesize
a solution that makes this problem tractable to solve in real time and
is applicable to a wide range of traffic scenarios.

\section{Overview of hierarchical distributed
  solution}\label{sec:solution-overview}

This section gives an outline of our hierarchical distributed solution
to the problem stated in Section~\ref{sec:problem-statement}. Our
algorithmic solution combines optimized planning and scheduling of
groups of vehicles with local distributed control to ensure safety and
execute the plans. Its three distinct aspects are:
\begin{enumerate}
\item grouping the vehicles into clusters,
\item scheduling the passage of the clusters through the intersection,
\item local vehicular control to achieve and maintain cluster
  cohesion, to avoid collisions, and to ensure the clusters meet the
  prescribed schedule.
\end{enumerate}
Each of these aspects is coupled with the other two. Moreover, an
overarching theme is the dynamic nature of the problem due to the
arrival and departure of vehicles. Any complete or partial solution
has to be computed as new vehicles come in (event based) or at regular
time intervals (time based).  In what follows, we provide a general
description of the main ingredients of each aspect. At any given time
$t$, we let $t_s$ be the latest time prior to $t$ at which the IM
samples the state of traffic and solves the corresponding static
scheduling problem. 

\subsubsection*{Aspect 1 -- generation of bubbles}

The primary motivation to cluster vehicles is to reduce the number of
independent entities that need to be considered in the
(computationally expensive) schedule optimization problem. For
instance, the maximum number of clusters can be fixed according to the
available computational resources so that the scheduling problem
remains tractable.  At time $t_s$, the vehicles present in the four
branches are grouped into $N$ clusters. We let $N_k$ denote the number
of clusters on branch~$k$. Given the position information of the
vehicles at $t_s$, we use $k$-means clustering on each branch
individually to identify the clusters.  The relative positions of the
vehicles of a cluster may vary significantly over the course of their
travel and the vehicles may not be in the form of a well-defined
platoon at all times. Hence, we refer to a cluster of vehicles as a
\emph{bubble} (shown as grey boxes in Figure~\ref{fig:intersect}). The
defining characteristic of a bubble is that all the vehicles of a
bubble cross the intersection together. The state of the
$i^{\text{th}}$ bubble is given by the tuple
\begin{equation*}
  \xi_i = (x_i, v_i, m_i, \bartauocc_i, \Ic_i) \in \real^4 \times 
  \{1, 2, 3, 4\},
\end{equation*}
%
%
where $x_i$, $v_i$ and $m_i$ are, respectively, the position of the
lead vehicle in the bubble, the velocity of the lead vehicle in the
bubble, and the number of vehicles in the bubble. We denote by
$\tauocc_i$, the \emph{occupancy time} of bubble $i$, which is the
time duration for which the intersection is occupied by bubble
$i$. The quantity $\bartauocc_i$ is an upper bound that can be
guaranteed \emph{a priori}, and is a function of the bubble size $m_i$
and various other system parameters. The quantity~$\Ic_i$ denotes
which of the four incoming branches the bubble is on. Within each
branch, we require the order of the bubbles to remain constant during
the bubbles' travel (i.e., there is no passing allowed).  To capture
the order of the bubbles on a branch, we define the function~$\Rc$,
\begin{align*}
  \Rc(i,q) \triangleq
  \begin{cases}
    1, &\text{if } \Ic_{i} = \Ic_{q},\ x_{q}(t_s) <
    x_{i}(t_s),
    \\
    &\nexists i_1 \text{ s.t. } \Ic_{i_1} = \Ic_i, \ x_{q}(t_s) <
    x_{i_1}(t_s) < x_{i}(t_s) ,
    \\
    0, &\text{otherwise} .
  \end{cases}
\end{align*}
According to this definition, $\Rc(i,q) = 1$ if and only if bubbles
$i$ and $q$ are on the same branch and bubble $q$ is the immediate
follower of bubble~$i$.  We describe in detail the generation of
bubbles and the algorithm to select the bubbles to schedule in
Section~\ref{sec:cluster-dyn} below. We impose a limit on the number
of bubbles that are scheduled at any given time to $\maxNbubbles$,
even if the actual number of bubbles in the system were greater, so as
to keep the computational cost manageable. However, in the algorithm
we describe in the sequel, each bubble is scheduled at least once and
some bubbles may be scheduled more than once. We let $t_{s_i}$ denote
the latest time prior to $t$ at which bubble $i$ was scheduled.

We index the vehicles in bubble $i$ as $(i,1), \ldots, (i, m_i)$,
where $(i,1)$ refers to the lead vehicle in bubble $i$ and so on until
$(i,m_i)$, the last vehicle in the bubble. We also find it convenient
for the label $(i,0)$ to represent the last vehicle $(i', m_{i'})$ of
the bubble $i'$ that precedes bubble $i$ on the same branch or, if
such bubble does not exist, we let $(i,0)$ be an imaginary vehicle
located at $\infty$. We drop the index~$i$ whenever there is no
ambiguity with regard to the bubble.

\subsubsection*{Aspect 2 -- scheduling of bubbles}

The job of the scheduler is to prescribe to each bubble an
\emph{approach time} $\tau_i$ - the time at which the $i^{\text{th}}$
bubble is to reach the beginning of the intersection, i.e.,
$x_i(\tau_i) = 0$, so that no two different bubbles collide. In
solving this problem, the scheduler has to respect the order of
bubbles on the same branch and take into account no-collision
constraints between bubbles on two different branches. The
preservation of the order of intersection crossing by the bubbles on
the same branch takes the form,
\begin{subequations}\label{eq:tau-constraints}
  \begin{align}
    &\tau_q \geq \tau_i + \bartauocc_i, \quad \text{if } \Rc(i,q) = 1,
  \end{align}
  for $i,q \in \until{N}$.  Note that these constraints only ensure
  that the passage of bubbles on a branch through the intersection
  occurs in the same order as they have arrived, but they do not
  necessarily exclude collisions for the entire travel time. The
  intra-branch collisions are avoided at a local level and we accept
  the resulting sub-optimality. On the other hand, the no-collision
  constraint between bubbles on two different branches takes the form,
  \begin{align}\label{eq:tau-diff-branch}
    &\tau_i \geq \tau_q + \bartauocc_q \ \text{OR} \ \tau_q \geq
    \tau_i + \bartauocc_i, \quad \text{if } \Ic_i \neq \Ic_q ,
  \end{align}
\end{subequations}
for $i,q \in \until{N}$.  The constraints~\eqref{eq:tau-diff-branch}
make the scheduling problem combinatorial in nature because of the
need to determine whether $i$ or $q$ goes first. Since the order on
each branch is to be preserved, the number of sub-problems is the
number of permutations of the multiset $\{ \Ic_k \}_{k=1}^N$, i.e.,
\begin{equation*}
  \frac{ N! }{ \prod_{k=1}^4 N_k ! } = 
  \frac{ \big( \sum_{k = 1}^4 N_k \big)! }{ \prod_{k=1}^4 N_k ! } ,
\end{equation*}
where recall that $N_k$ is the number of bubbles on branch $k$ and $N$
is the total number of bubbles.  We describe in detail the algorithm
for optimal scheduling of bubbles in Section~\ref{sec:schedule}.

\subsubsection*{Aspect 3 -- local vehicular control}

The local vehicular control has various equally relevant goals. The
first goal is to avoid collisions within each bubble and among
different bubbles in the same branch. The second goal is for the local
vehicular control to ensure that the bubble approaches the
intersection at the prescribed time $\tau_i$ and that the occupancy
time of the bubble, $\tauocc_i$, is no more than $\bartauocc_i$. The
scheduler requires the quantity $\bartauocc_i$ and other quantities
such as earliest and latest times of approach at the intersection for
the bubble that are functions of the initial conditions. All these
quantities may be computed by the bubble and passed on to the IM or,
instead, the state of each car may be passed to the IM.  We assume
that the control law at the vehicle level ensures that a vehicle does
not change bubbles during the course of its travel time. Thus, as far
as the scheduling aspect is concerned, $m_i$ may be assumed constant
in time. We describe in detail the local vehicular control component
in Section~\ref{sec:vehicle-control} below.

\begin{remark}\longthmtitle{Relaxation of
    assumptions}\label{re:assumptions}
  {\rm We discuss here to what extent the assumptions made in
    Section~\ref{sec:problem-statement} can be relaxed in our proposed
    design. We make assumptions (ii) and (iv) only for the sake of
    simpler notation and ease of exposition. Our algorithm can handle
    non-identical vehicles with differing dimensions and differing
    acceleration limits, though those quantities need to be
    known. Simultaneous use of the intersection by vehicles on
    compatible branches/directions is definitely possible in our
    framework and indeed makes the scheduling problem easier. We can
    relax assumption (v) if the sources or sinks are not close to the
    intersection with minor changes in our algorithm for bubble
    generation. We can avoid assumption (iii) and allow turning within
    our framework. However, the differing travel speeds when turning
    and going straight affects the computation of the intersection
    occupancy time, which might make the design conservative.  We
    believe this conservativeness could be addressed by relaxing
    assumption (i) and incorporating multiple lanes into the design.}
  \oprocend
\end{remark}

\section{Dynamic vehicle clustering}\label{sec:cluster-dyn}

The primary motivation for clustering vehicles into bubbles is to
reduce the computational burden on the scheduler. Consequently, we
impose the upper bound $\maxNbubbles$ on the number of bubbles that
the scheduler needs to consider at any given instance. Further, as
new vehicles arrive, they need to be assigned to new bubbles. In order
to balance both requirements, we divide each branch into three zones,
as shown in Figure~\ref{fig:zones}: staging zone (of length $L_s$),
mid zone (of length $L_m$) and exit zone (of length $L_e$). For each
branch $k \in \branches$, we let $\stagingset_k$, $\midset_k$ and
$\exitset_k$ be the set of positions on the branch $k$ corresponding
to the staging, mid and exit zones, respectively.
\begin{figure}[!htpb]
  \centering
  \includegraphics[width=0.9\linewidth]{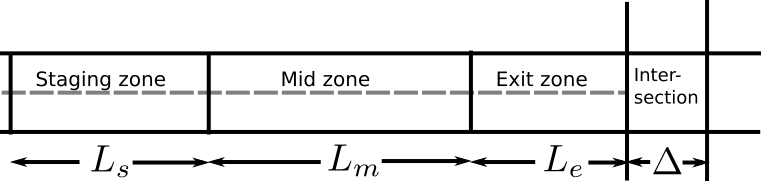}
  \caption{Division of an incoming branch into
    zones.}\label{fig:zones}
\end{figure}

The \clusteralgo algorithm is executed every $\timestepCS$ units of
time. At each clustering instance $t_s = s \timestepCS$, $s \in
\integersnonnegative$, the vehicles in the staging zone that do not
already belong to a bubble are clustered. Thus, the choice
$\timestepCS < \frac{ L_s }{ v^M }$, where recall that $v^M$ is the
max speed limit, ensures that every vehicle belongs to a bubble before
it leaves the staging zone and enters the mid zone. We impose an upper
bound $\Ncbar_k$ on the number of new bubbles that may be created on
branch $k$ at any clustering instance. At a clustering instance $t_s =
s \timestepCS$, let $\Nnewvehicles_k$ denote the number of vehicles to
be clustered in the staging zone of branch~$k$. Then, the
$\Nnewvehicles_k$ vehicles are clustered based on their position using
the $\Mc_k$-means algorithm, with $\Mc_k = \min \{ \Nnewvehicles_k,
\Ncbar_k \}$. Thus, the $\Nnewvehicles_k$ on branch $k$ are
partitioned into $\Mc_k$ number of clusters or bubbles such that the
sum of squares of the distances from each car to the center of its
bubble is minimized, see e.g.,~\cite{SPL:82}.  The clustering
component in our design is modular and hence any clustering algorithm
that is well suited may be used.

The algorithm also makes sure that no more than~$\maxNbubbles$ bubbles
are passed to the IM manager for scheduling at any instance. This is
achieved using two observations. First, previously scheduled bubbles
that have already entered the exit zone of their branch are no longer
fed to the IM for scheduling (i.e., its schedule is not modified any
further). Second, if the number of newly created bubbles and the
previously created bubbles yet to enter the exit zone exceeds
$\maxNbubbles$, then the algorithm pops out the required number of
bubbles from the top of the list of bubbles previously scheduled
(corresponding to the ones closer to their respective exit zones). We
present the precise description of the \clusteralgo algorithm in
Algorithm~\ref{algo:dyn-CS}.

\begin{algorithm}[htb]
  {
    \footnotesize \vspace*{1ex} 
    \begin{algorithmic}[1]
      \REQUIRE 
      $\Lc_p, \taumin_p$
      \\
      \mycomment{Ordered list of bubbles scheduled at $(s-1)
        \timestepCS$ and earliest approach time used in scheduling
        them}
      \STATE $\Lc \gets \Lc_p \setminus \{ j \in \Lc : \Ic_j = k \
      \land \ x_j \notin \stagingset_k \union \midset_k \}$
      \\
      \mycomment{remove bubbles that are not completely within the
        staging and the mid zones}
      \FOR{$k=1$ \TO $4$}
      \STATE $\Ncbar_k$
      \mycomment{max new bubbles on branch $k$}
      \STATE $\Mc_k \gets \min\{ \Nnewvehicles_k, \Ncbar_k \}$
      \mycomment{\# new bubbles on branch $k$}
      \STATE Cluster new vehicles on branch $k$ using $\Mc_k$-means
      algorithm \label{step:k-means}
      \ENDFOR
      \STATE $\displaystyle \Mc \gets \sum_{k=1}^4 \Mc_k$
      \IF{$\Mc + \card{\Lc} > \maxNbubbles$}
      \STATE Remove first $\Mc + \card{\Lc} - \maxNbubbles$ bubbles
      from $\Lc$
      \mycomment{Ensure only $\maxNbubbles$ bubbles provided to
        scheduler by dropping the earliest bubbles in previous
        schedule}
      \ENDIF
      \STATE Append new bubbles to $\Lc$
      \STATE $\taumin \gets \max \left( \{ \taumin_p \} \union
        \{\tau_i + \bartauocc_i : i \in \Lc_p \setminus \Lc \}
      \right)$
      \\
      \mycomment{earliest approach time for the bubbles in~$\Lc$}
      \ENSURE $\Lc$, $\taumin$
    \end{algorithmic}	   
  }
  \caption{\hspace*{-.5ex}: \clusteralgo at~$s
    \timestepCS$}\label{algo:dyn-CS}
\end{algorithm}

The algorithm takes in the list of bubbles $\Lc_p$ scheduled on the
last iteration and an earliest approach time $\taumin_p$ used when
scheduling~it. The output is a list of bubbles $\Lc$ to be scheduled
and the earliest approach time $\taumin$ for them. Note from step 12
of Algorithm~\ref{algo:dyn-CS} that $\taumin$ is an upper bound on the
time by which all the bubbles not in the $\Lc$ list are guaranteed to
cross the intersection. Thus, when scheduling $\Lc$, the scheduler
imposes the constraint that the bubbles in $\Lc$ approach the
intersection no earlier than~$\taumin$.

\begin{remark}\longthmtitle{Effect of zone lengths on clustering and
    scheduling} 
  {\rm The lengths of the three zones illustrated in
    Figure~\ref{fig:zones} directly affect the resulting traffic
    coordination. Although we do not pursue here a systematic design
    of these zone lengths, we can identify some basic observations of
    their effect on clustering and scheduling.  We envision these zone
    lengths to be of the order of several tens of meters. The length
    $L_s$ of the staging zone has a direct effect on the time step of
    the periodic execution of clustering and scheduling as well as on
    the number of vehicles per bubble. The length $L_m$ of the mid
    zone has an effect on the likelihood of revising a bubble's
    schedule on the next iteration. Finally, the length $L_e$ of the
    exit zone has an effect on the feasibility of the scheduling
    problem, which we guarantee by assuming that $L_e$ is large enough
    for a vehicle to come to a complete stop from a maximum speed of
    $v^M$ in under a distance~$L_e$. }  \oprocend
\end{remark}

\begin{remark}\longthmtitle{Re-clustering}
  {\rm The \clusteralgo algorithm is just one method of defining
    bubbles and selecting which ones to select. In this algorithm, a
    vehicle is assigned to a bubble only once and the vehicle is part
    of that bubble through out its travel. However, one could
    implement a strategy which re-clusters all vehicles in the staging
    and mid zones so that vehicles may be reassigned to a different
    bubble, bubbles may be merged or split as needed, and so on. Such
    an algorithm would also allow sources and sinks on the branch such
    as smaller streets, homes, and retail.} \oprocend
\end{remark}

\section{Scheduling of bubbles}\label{sec:schedule}

This section describes the scheduling algorithm employed by the
intersection manager (IM) to decide the order of passage through the
intersection of the bubbles in $\Lc$ provided by the clustering
algorithm. The scheduling algorithm is also executed
every~$\timestepCS $ units of time. In this section, we let $\Lc$ be
the set $\until{N}$, where $N = \card{\Lc}$, without loss of
generality.

\subsection{Cost function and constraints}\label{sec:simplify-cost}

In our approach, the IM schedules bubbles as a whole using an
abstraction of the vehicle dynamics and the cost function. First,
regarding the vehicle dynamics, we note that the inter-vehicle
approach times at the intersection and the resulting occupancy time of
a bubble is a degree of freedom. However, we have made the alternative
choice of not considering it as such in the scheduling algorithm, and
instead only use an upper bound on the occupancy time $\bartauocc_i$
(that the local vehicular control component can guarantee) appearing
in the constraints~\eqref{eq:tau-constraints}.  Second, regarding the
cost function, we abstract the fuel cost for the vehicles in a bubble
$i$ into a single function $F_i$ that depends only on the average
velocity of the bubble $i$ (lead vehicle in the bubble) for $t \in
[t_s, t_s + \tau_i]$, where $t_s = s \timestepCS$ is the time at which
the scheduling algorithm is executed.  Thus, the scheduling algorithm
minimizes the following simplified cost function~$\Cc \triangleq
\Cc_\Lc$ where $\Cc_\Pc$ for a given list of bubbles $\Pc$ is
\begin{align}\label{eq:simp-cost-fun}
  \Cc_\Pc &\triangleq \sum_{i \in \Pc} m_i ( W_T \tau_i +
  F_i(\avgvel_i) ) \notag
  \\
  &= \sum_{i \in \Pc} m_i \Big( W_T \frac{d_i}{\avgvel_i} +
  F_i(\avgvel_i) \Big) \triangleq \sum_{i \in \Pc} \phi_i(\avgvel_i) ,
\end{align}
where $\avgvel_i$ is the average velocity of the lead vehicle in
bubble $i$ for $t \in [t_s, t_s + \tau_i]$, i.e.,
$\avgvel_i = \frac{ d_i }{ \tau_i }$, where $d_i \triangleq -
x_i(t_s)$. The optimization variables are $\avgvel_i$ 
for each bubble~$i \in \Lc$.

Note that in the cost function $\Cc$, the functions $F_i$ could, in
general, depend on initial conditions modeled as parameters - such as
the distance $d_i$ to reach the intersection.  The cost
function~\eqref{eq:simp-cost-fun} models a combination of cumulative
travel time and total fuel usage.  Motivated by the fact that fuel
efficiency is typically an increasing function of vehicle speed for
speeds under the limits enforced on most roads with intersections, we
make the assumption that, for each $i \in \Lc$, $F_i: [ 0, v^M ]
\mapsto \realpositive$ is monotonically decreasing.

Regarding the constraints, conditions on the travel times can be
re-expressed as conditions on average velocities as
\begin{align}\label{eq:cji-bji}
  &\tau_i \geq \tau_q + \bartauocc_q \iff \frac{ d_i }{ \avgvel_i }
  \geq \frac{ d_q }{ \avgvel_q } + \bartauocc_q \notag
  \\
  & \iff \avgvel_q \geq c_{qi} \avgvel_i + b_{qi} \avgvel_q \avgvel_i,
  \quad c_{qi} = \frac{d_q}{d_i}, \ b_{qi} = \frac{ \bartauocc_q }{
    d_i } .
\end{align}
Thus, we re-express the no-collision
constraints~\eqref{eq:tau-constraints} as
\begin{subequations}\label{eq:v-constraints}
  \begin{align}
    &\avgvel_i \geq c_{iq} \avgvel_q + b_{iq} \avgvel_i \avgvel_q,
    \quad \text{if } \Rc(i,q) = 1, \label{eq:v-constraints-a}
    \\
    &\avgvel_q \geq c_{qi} \avgvel_i + b_{qi} \avgvel_q \avgvel_i \
    \text{OR} \ \avgvel_i \geq c_{iq} \avgvel_q + b_{iq} \avgvel_i
    \avgvel_q, \; \text{if } \Ic_i \neq \Ic_q . \label{eq:v-constraints-b}
  \end{align}
  In addition, we also need to ensure that the scheduling at instance
  $s \timestepCS$ of the bubbles in $\Lc$ does not conflict with the
  ones that have been previously scheduled. Formally, this corresponds
  to having the time $\tau_i$ to reach the intersection for bubble $i$
  be no less than $\taumin$ (cf.  step 12 of
  Algorithm~\ref{algo:dyn-CS}). Equivalently, we require
  \begin{equation}\label{eq:v-constraints-c}
    \avgvel_i \leq \frac{ d_i }{ \taumin } .    
  \end{equation}
\end{subequations}

Note that the scheduling problem is combinatorial in nature due to the
no-collision constraints~\eqref{eq:v-constraints-b}. Thus, even though
the cost function $\Cc$ is simple and the optimization variables are
the average velocities $\avgvel_i$, we believe this formulation
provides a good balance between usefulness and computational
tractability. Further, the local vehicular control we present in
Section~\ref{sec:vehicle-control} seeks an optimal control profile to
achieve the prescribed average velocity for the bubble, which
justifies the restriction to $\avgvel_i$ as the optimization variables
in the scheduling aspect. Thus our proposed solution, although
sub-optimal, is still principled.

We next describe our solution to the scheduling problem consisting of
minimizing~$\Cc = \Cc_\Lc$ in~\eqref{eq:simp-cost-fun} under the
constraints~\eqref{eq:v-constraints} and $\avgvel_i \in [ \lowvel_i,
\hivel_i ]$.  The lower $\lowvel_i \ge 0$ and upper $\hivel_i \le v^M$
limits on the average velocity depend on the initial conditions of the
vehicles and desired speed limits. The quantities $\lowvel_i$ and
$\hivel_i$ are inversely related to the \emph{latest time of approach}
and the \emph{earliest time of approach} at the intersection for
bubble $i$, respectively. The computation of these quantities is
described in Section~\ref{sec:hivel-lowvel}. Similarly, the upper
bound $\bartauocc_i$ on the occupancy times may be computed as in
Section~\ref{sec:bartauocc}.  In the first part of our solution to the
scheduling problem, we determine the optimal schedule and optimal cost
given a fixed order of bubble passage through the intersection. In the
second part, we use a branch-and-bound algorithm to find the optimal
order and schedule.

\subsection{Optimal bubble average velocity for fixed order of
  passage}

Here we address the problem of determining, given a desired order of
bubble passage through the intersection, the optimal average
velocities of the bubbles and the associated optimal cost. For this
purpose, define an \emph{order} of the approach times of the bubbles
as a permutation, $P$, of the integers from $1$ to $\card{P} \leq N$.
We use $P(i)$ to denote the $i^{\text{th}}$ element in the order, with
the bubble $P(1)$ passing through the intersection first and so on. We
use $\sigP{P}(i)$ to denote the position of bubble $i$ in the
order~$P$. Clearly, for a permutation to respect the intra-branch
orders, $\sigP{P}(i) < \sigP{P}(q)$ if $\Rc(i,q) = 1$. Given~$P$
respecting the intra-branch orders, the \veloptalgo algorithm,
formally described in Algorithm~\ref{algo:optim-given-P}, finds a
solution to the optimization of $\Cc_P$ under the
constraints~\eqref{eq:v-constraints}, $\avgvel_i \in [ \lowvel_i,
\hivel_i ]$, and with order~$P$.

\begin{algorithm}[htb]
  {
    \footnotesize \vspace*{1ex} 
    \begin{algorithmic}[1]
      \REQUIRE Order $P$
      \STATE $C \gets 0$
      \FOR{$h=1$ \TO $\card{P}$} 
      \STATE $i \gets P(h)$
      \mycomment{bubble $i$ is in position $h$ in $P$}
      \IF{$h = 1$}
      \STATE $\avgvel_i^P \gets \hivel_i$
      \ELSE
      \STATE $q \gets P(h-1)$
      \mycomment{bubble $q$ is in position $h-1$ in $P$}
      \STATE $\avgvel_i^P \gets \min \{ \hivel_i, \frac{ \avgvel_q^P }{
        c_{qi} + b_{qi} \avgvel_q^P } \}$
      \ENDIF
      \mycomment{$\avgvel_i^P$ is the optimizer for bubble $i$}
      \STATE $C \gets C + \phi_i(\avgvel_i^P)$
      \mycomment{update cost}
      \ENDFOR      
    \end{algorithmic}	   
  }
  \caption{\hspace*{-.5ex}: \veloptalgo}\label{algo:optim-given-P}
\end{algorithm}

The following result shows that, for an order that respects the
intra-branch order, the algorithm finds the average velocities that
optimize the cost function $\Cc_P$.

\begin{lemma}\longthmtitle{\veloptalgo algorithm
    optimizes the schedule given an overall order that respects the
    intra-branch orders}
  For each $i \in \until{N}$, assume the fuel cost function $F_i$ is
  monotonically decreasing.  Let $P$, with $|P| \le N$, be an order
  respecting the intra-branch orders and denote by $\avgvel^P =
  (\avgvel_1^P, \ldots, \avgvel_N^P)$ and $C$ the output of
  Algorithm~\ref{algo:optim-given-P}. Then, $\avgvel^P$ and $C$ are,
  respectively, the minimizer and the minimum cost of the optimization
  problem with the cost function as $\Cc_P$~\eqref{eq:simp-cost-fun}
  under the constraints~\eqref{eq:v-constraints}, $\avgvel_i \in [
  \lowvel_i, \hivel_i ]$.
\end{lemma}
\begin{proof}
  Given the order $P$, the constraints~\eqref{eq:v-constraints}
  reduce to
  \begin{equation*}
    \frac{ \avgvel_q }{ c_{qi} + b_{qi} \avgvel_q } \geq \avgvel_i
  \end{equation*}
  where $q = P(h-1)$, $i = P(h)$ and $h \in \{2, \ldots, N\}$. The
  left-hand side of the inequality is an increasing function of
  $\avgvel_j$. Further since $F_i$ is a monotonically decreasing
  function for each $i$, $\avgvel_i^P$ takes the maximum possible
  value. The algorithm computes the components of $\avgvel^P$
  iteratively and the result follows.
\end{proof}

\subsection{Optimal ordering via branch-and-bound}

We propose a branch-and-bound algorithm to solve the optimal
scheduling problem.  We start by providing an informal description.
\begin{quote}
  \emph{Informal description:} A branch-and-bound algorithm consists
  of a systematic enumeration of the set of candidate solutions as a
  rooted tree, with the full set at the root.  The algorithm explores
  branches of the tree, which represent subsets of the set of
  candidate solutions.  Before enumerating the candidate solutions of
  a branch, the branch is checked against upper bounds on the optimal
  solution, and is discarded if it is determined that it cannot
  produce a better solution than the best one found so far.
\end{quote}

We formally specify each of the components in this description next,
starting with the rooted tree.  We let $\Pc$ denote any ordered list
of up to length $N$, with non-repeating numbers drawn from $\{1,
\ldots, N\}$, and preserving the individual branch orders. With this
notation, the empty list $\Pc = \emptyset$ denotes the root of the
tree, representing all feasible orders. Similarly, $\Pc = (i_1,
\ldots, i_h)$ denotes the subtree of all the feasible orders in which
bubble $i_1$ crosses the intersection first, $i_2$ second, and so on,
until bubble $i_h$ is the $h^\text{th}$ to cross, with the order of
the remaining bubbles undetermined.

Our next step is to provide a way to determine a lower bound on the
achievable optimal value of any given branch. This follows from the
observation that (i) the execution of the \veloptalgo algorithm finds
the optimal value of the average velocity a bubble given the order of
\emph{all} the bubbles preceding it, but (ii) one can compute an upper
bound for the optimal value even if only part of the order of bubbles
preceding it is known.  The description in
Algorithm~\ref{algo:update-hilim} of this procedure, termed
\velboundalgo algorithm, relies on four ordered lists, termed queues,
one for each branch.
\begin{algorithm}[htb]
  {
    \footnotesize \vspace*{1ex} 
    \begin{algorithmic}[1]
      \STATE $l \gets \last{\Pc}$
      \mycomment{$l$ is last bubble in $\Pc$}
      \STATE Compute $\avgvel_l^{\Pc}$ using \veloptalgo with input
      $\Pc$
      \FOR{$k=1$ \TO $4$}
      \STATE $\Qc_k \gets Q_k \setminus \Pc$
      \mycomment{pop-out $\Pc$ from $Q_k$}
      \IF{$\Qc_k \neq \emptyset$}
      \STATE $i \gets \first{\Qc_k}$
      \mycomment{$i$ is first of remaining bubbles in  $\Qc_k$}
      \STATE $\hilim{\Pc}_i \gets \min \{ \hivel_i, \frac{
        \avgvel_l^{\Pc} }{ c_{li} + b_{li} \avgvel_l^{\Pc} } \}$
      \FOR{$s=2$ \TO $\Len{\Qc_k}$} 
      \STATE $i \gets \Qc_k(s)$ 
      \STATE $q \gets \Qc_k(s-1)$
      \STATE $\hilim{\Pc}_i \gets \min \{ \hivel_i, \frac{
        \hilim{\Pc}_q }{ c_{qi} + b_{qi} \hilim{\Pc}_q } \}$
      \ENDFOR
      \ENDIF
      \ENDFOR
    \end{algorithmic}	   
  }
  \caption{\hspace*{-.5ex}: \velboundalgo}\label{algo:update-hilim}
\end{algorithm}
The queue for branch $k$, $Q_k = ( i_{k,1}, \ldots, i_{k,N_k})$, is
initialized to the list of all the bubbles on branch $k$ in their
order of arrival (thus $\Rc(i_{k,q}, i_{k,q+1}) = 1$ for all $q \in \{
1, \ldots, N_k - 1\}$).  We denote by $\hilim{\Pc}_i$ the upper bound
on the average velocity $\avgvel_i$ of bubble $i$ obtained by
Algorithm~\ref{algo:update-hilim} given that a non-empty $\Pc$
precedes it.  This allows us to lower bound the optimal cost for any
order in the subtree $\Pc$ in terms of $\avgvel_i^{\Pc}$ and
$\hilim{\Pc}_i$ as follows,
\begin{equation} \label{eq:bnd-optim-cost-subtree} C^{\Pc} \triangleq
  \sum_{i \in \Pc} \phi_i( \avgvel_i^{\Pc} ) + \sum_{i \in \Lc
    \setminus \Pc} \phi_i( \hilim{\Pc}_i ).
\end{equation}
This lower bound is precisely what is required to implement a
branch-and-bound algorithm to find the optimal schedule for the
bubbles.

Specifically, the branch-and-bound algorithm starts by picking an
arbitrary candidate order and computing the cost for it, using the
\veloptalgo algorithm, and storing the two as the current best
solution and cost. Then, starting at the root node of the tree of all
feasible orders, the algorithm searches (e.g., using depth-first or
breadth-first search) for an optimal solution. If at any time a leaf
node, which corresponds to a fully determined order, is reached and
its cost is better than the current best, then the current best
solution and cost are updated. For any other node $\Pc$ in the
tree,~\eqref{eq:bnd-optim-cost-subtree} provides a lower bound
$C^{\Pc}$ on the cost of all the orders represented by the node $\Pc$.
If $C^{\Pc}$ is greater than the current best known cost, then the
subtree $\Pc$ is discarded. This process continues until the algorithm
finds the optimal solution. We refer to this process as the \schedalgo
algorithm.

\section{Local vehicular control}\label{sec:vehicle-control}

The local vehicular control component of our hierarchical-distributed
coordination approach involves two main tasks: (i) compute, for each
bubble $i$, the lower $\lowvel_i$ and upper $\hivel_i$ average
velocity bounds, and the upper bound on the intersection occupancy
time $\bartauocc_i$ that are provided to the scheduler; and (ii)
control the vehicles ensuring no collisions and that all the vehicles
of bubble $i$ cross the intersection within the time interval
$[\tau_i, \tau_i + \bartauocc_i]$ prescribed by the scheduler. The
successful execution of each of these tasks requires an understanding
of the vehicle dynamics and the desired safety constraints and the
effect of each on the other.  We achieve this by resorting to the
controller design in our previous work~\cite{PT-JC:17-tcns} on vehicle
strings under finite-time and safety specifications. For the sake of
completeness, we review here the main elements and performance
guarantees of this design as needed by our overarching
hierarchical-distributed approach.

\subsection{Bounds on average bubble velocity}\label{sec:hivel-lowvel}

Recall that $\avgvel_i$ is the average velocity of the lead vehicle of
bubble $i$ from $t_s$ and until the lead vehicle is supposed to reach
the beginning of the intersection at $\tau_i$. Thus, it would seem
that computing lower and upper bounds on the achievable average
velocity of the lead vehicle in the bubble is sufficient to determine
$\hivel_i$ and $\lowvel_i$. However, ignoring the initial conditions
of the other vehicles in the bubble in the computation of $\hivel_i$
and $\lowvel_i$ poses the risk of lengthening the guaranteed upper
bound $\bartauocc_i$ on the occupancy time. The reasoning for this is
better explained in terms of earliest times of approach at the
intersection of the vehicles.

In bubble $i$, we let $\taue_{i,j}$ be the earliest time vehicle
$(i,j)$ can reach the intersection ignoring the other vehicles on the
branch. Letting $t_{s_i} = s_i \timestepCS$ be the time at which
bubble $i$ was last scheduled, the quantity $\taue_{i,j} - t_{s_i}$ is
then the time it takes $\xv_{i,j}$ to reach $0$ from
$\xv_{i,j}(t_{s_i})$ for the trajectory with maximum acceleration
until $\vv_{i,j} = v^M$ and zero acceleration thereafter. Thus, we see
that if $\taue_{i,j}$ for some $j > 1$ is significantly greater than
$\taue_{i,1}$ then the vehicle $(i,1)$ has to slow down to approach
the intersection at a time later than $\taue_{i,1}$ so that the
guaranteed upper bound $\bartauocc_i$ on the occupancy time is small
enough.  Thus, we propose an alternative solution. To do so, we first
introduce the notion of safe-following distance.

\begin{definition}\longthmtitle{Safe-following
    distance}\label{dfn:sfd}
  The maximum braking maneuver (MBM) of a vehicle is a control action
  that sets its acceleration to $u_m$ until the vehicle comes to a
  stop, at which point its acceleration is set to $0$ thereafter.  Let
  $j-1$ and $j$ be the indices of two vehicles on the same branch,
  with vehicle $j$ immediately following $j-1$, and define
  \begin{multline}\label{eq:sf-dist}
    \Dc(\vv_{j-1}(t), \vv_j(t)) =
    \\
    L + \max \Big\{ 0, \frac{ 1 }{ -2u_m } \left( (\vv_{j}(t))^2 -
      (\vv_{j-1}(t))^2 \right) \Big\} .
  \end{multline}
  The quantity $\Dc(\vv_{j-1}(t), \vv_j(t))$ is \emph{a safe-following
    distance} at time~$t$ for the pair of vehicles $j-1$ and $j$
  because if $\xv_{j-1}(t) - \xv_j(t) \geq \Dc(\vv_{j-1}(t),
  \vv_j(t))$ and, if each of the two vehicles were to perform the MBM,
  then the two vehicles would be safely separated, $\xv_{j-1} - L \geq
  \xv_j$ until they come to a stop. \oprocend
\end{definition}

Now, assuming a nominal speed $\nunom$ for vehicles when entering the
intersection, we set $\Dcnom \triangleq \Dc(\nunom, v^M)$, which has
the connotation of a safe inter-vehicle distance given a vehicle is
traveling at the maximum allowed speed $v^M$ and the vehicle preceding
it traveling at a speed greater than or equal to $\nunom$. Then, we
also define $\Tnom \triangleq \Dcnom / \nunom$ as the \emph{nominal
  inter-vehicle approach time}. With this nominal inter-vehicle
approach times of vehicles in a bubble we see that the earliest time
of approach for vehicle $(i,j)$ forces the earliest time of approach
of bubble $i$, i.e. vehicle $(i,1)$, is no less than $\taue_{i,j} -
(j-1) \Tnom$. Hence, we define \emph{earliest time of approach} for
the bubble $i$, $\taue_i$ as
\begin{equation}\label{eq:taum}
  \taue_i \triangleq \max \{ \taue_{i,j} - (j-1) \Tnom :
  j \in \{1, \ldots, m_i\} \} ,
\end{equation}
and let $\hivel_i = \frac{ -x_i(t_s) }{ \taue_i }$.  Analogous
computations with maximum deceleration yield the \emph{latest time of
  approach} $\taul_i$ of bubble $i$, possibly with $\taul_i = \infty$,
and the corresponding lower bound $\lowvel_i \geq 0$ for the average
velocity. Hence, the values we obtain in this way for $\lowvel_i$ and
$\hivel_i$ are, respectively, larger and smaller than the ones we
would have obtained if we only took into account the lead vehicle of
the bubble.

For a given upper bound on the occupancy time and the sets of
$\lowvel_i$ and $\hivel_i$ for $i \in \Lc$, a feasible schedule might
not always exist. Thus, to guarantee the feasibility of the scheduling
problem in a simple fashion, we assume that the exit zone length $L_e$
is large enough. 

\begin{lemma}\longthmtitle{Existence of a feasible
    schedule}\label{lem:zone-length-assump}
  If the exit zone length, $\displaystyle L_e \geq \frac{ (v^M)^2 }{
    -2u_m } + \frac{ (\nunom)^2 }{ 2u_M }$, then there always exists a
  feasible schedule with which each vehicle is able to enter the
  intersection with a speed of at least $\nunom$.
\end{lemma}
\begin{proof}
  Recall, that a schedule to a bubble is assigned when all the
  vehicles in the bubble are still in the staging or the mid zones.
  Clearly, the condition on $L_e$ implies that any vehicle in the
  staging zone or the mid zone ($\xv_j \leq - L_e$) can come to a
  complete stop and then accelerate to a speed of at least $\nunom$
  before arriving at the beginning of the intersection ($\xv_j =
  0$).
\end{proof}

\subsection{Vehicle controller design}

The scheduler prescribes for each bubble a time at which the vehicles
in the bubble may start to cross the intersection. The local vehicular
control must ensure that the vehicles of bubble $i$ start and finish
crossing the intersection within the time interval $[\tau_i, \tau_i +
\bartauocc_i]$ while respecting the safety
constraints~\eqref{eq:sf-dist}. In this section, we describe an
algorithm to achieve this task. The algorithm has three main parts:
(i) an uncoupled controller ensuring that the vehicle arrives at the
intersection at a designated time if the presence of all other
vehicles is ignored. This controller is applied when the preceding
vehicle is sufficiently far in front, (ii) a safe-following controller
ensuring that the vehicle follows the preceding vehicle safely when
the latter is not sufficiently far in front; and (iii) a rule to
switch between the two controllers.

\subsubsection{Uncoupled controller}

For each vehicle $j \in \{1, \ldots, m_i\}$ in bubble $i$, we define,
\begin{equation}\label{eq:tau-ik}
  \tau_{i,j} \triangleq \tau_i + (j - 1) \Tnom .
\end{equation}
Given the constraints that the scheduler takes into account, we have
$\tau_i \in [ \taue_i, \taul_i ]$. This, together
with~\eqref{eq:taum}, implies that $\tau_{i,j} \in [ \taue_{i,j},
\taul_{i,j} ]$. Now, let
\begin{equation*}
  (t, \xv_{i,j}, \vv_{i,j}) \mapsto \guc(\tau_{i,j}, t, \xv_{i,j},
  \vv_{i,j}) 
\end{equation*}
be a feedback controller that ensures $\xv_{i,j}(\tau_{i,j}) = 0$ for
the dynamics~\eqref{eq:vehicle-dyn} starting from the current state
$(\xv_{i,j}(t), \vv_{i,j}(t))$ at time $t$ (assuming feasibility),
respecting the control and velocity constraints, but not necessarily
the inter-vehicle safety constraints. We refer to it as the
\emph{uncoupled} controller.  Such a controller exists for each
vehicle at least at $t = t_{s_i}$, where $t_{s_i} = s_i \timestepCS$
is the time at which bubble $i$ was scheduled, due to the fact that
$\tau_{i,j} \in [ \taue_{i,j}, \taul_{i,j} ]$. Here, we take as $\guc$
the optimal feedback controller that generates velocity profiles as
shown in Figure~\ref{fig:pwc-control} obtained by optimizing
\begin{align*}
  \int_{t}^{\tau_j} |\uv_j(s)| ds
\end{align*}
with optimization variables $a_1$, $a_2$ (the areas of the indicated
triangles), $\vv_j(\tau_j)$, $\nu^l$ and $\nu^u$, where we have
dropped the bubble index $i$. The constraints are $\nu^l \in [0,
\vv_j(t)]$, $\nu^u \in [\vv_j(t), v^M]$, $\vv_j(\tau_{i,k}) \in
[\nunom, v^M]$, $a_1, a_2 \geq 0$ and that the total area under the
curve must be equal to $-\xv_j(t)$. The feedback controller may be
found by tabulating the optimal control solution.

\begin{figure}[!htb]
  \centering \subfigure[\label{fig:pwc_down}]
  {\includegraphics[width=0.35\textwidth]{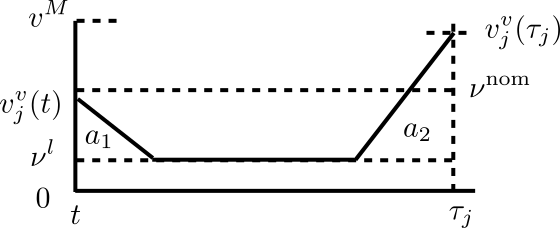}}

  \subfigure[\label{fig:pwc_up}]
  {\includegraphics[width=0.35\textwidth]{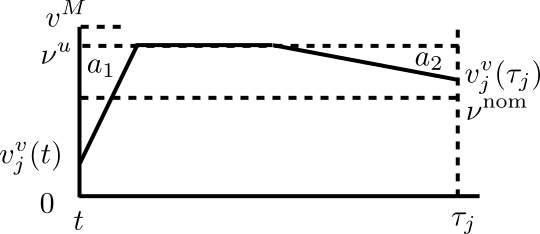}}
  \vspace*{-1ex}
  \caption{Candidate velocity profiles to obtain $\guc$, which takes
    form (a) or (b) depending on the velocity $\vv_j(t)$, $\nunom$,
    $v^M$, $\tau_j$ and the distance to
    go~$-\xv_j(t)$. } \label{fig:pwc-control}
\end{figure}


Note that the control $\guc$ assumes the presence of no other vehicles
on the branch. Thus, the actual approach time, $\Ta_{i,j}$, of the
vehicle $(i,j)$ may be later than $\tau_{i,j}$. At time $t_{s_i}$,
when the bubble is scheduled, an optimal control $\guc$ does exist for
each of its vehicles because of the way the times $\tau_{i,j}$ are
defined in~\eqref{eq:tau-ik}. However, at a future time $t$, such a
feasible $\guc$ might not exist because the vehicle is slowed down by
preceding vehicles and no control exists to ensure $\Ta_{i,j} =
\tau_{i,j}$ along with the other constraints. Additionally, for $t >
\Ta_{i,j}$, i.e., after the vehicle enters the intersection, the
optimal controller is not well defined and does not exist.  As a
shorthand notation, we use $\exists \Fc_{i,k}$ (respectively $\nexists
\Fc_{i,k}$) to denote the existence (respectively, lack thereof) of an
optimal control $\guc$. In order for the control $\guc$ to be well
defined at all times, we let
\begin{equation*}
  \guc(\tau_{i,k}, t, \xv_{i,k}, \vv_{i,k}) \triangleq
  u_M, \quad \text{if } \nexists \Fc_{i,k} .
\end{equation*}

\subsubsection{Controller for safe following}

As mentioned earlier, this controller is applied only when a vehicle
is sufficiently close to the vehicle preceding it. Besides maintaining
a safe-following distance, the controller must also ensure that the
resulting evolution of the vehicles in the bubble $i$ is such that the
occupancy time is no more than $\bartauocc_i$. Here, we present a
design to achieve these goals. For a pair of vehicles $j-1$ and $j$,
with $j$ following $j-1$, we define the \emph{safety ratio} as
\begin{equation}\label{eq:sigma}
  \sigma_j(t) \triangleq \frac{ \xv_{j-1}(t) - \xv_j(t) }{ 
    \Dc(\vv_{j-1}(t), \vv_j(t)) } ,
\end{equation}
which is the ratio of the actual inter-vehicle distance to the
safe-following distance. Hence, we would like to maintain this
quantity above $1$ at all times. Notice from~\eqref{eq:sf-dist} that
if $\vv_{j-1}(t) > \vv_j(t)$, then $\sigma_j$ increases and safety is
guaranteed. Thus, it is sufficient to design a controller that ensures
safe following when $\vv_j(t) \geq \vv_{j-1}(t)$. For vehicle $j$, we
denote $\zeta_j \triangleq ( \vv_{j-1}, \vv_j, \sigma_j )$. Define the
\emph{unsaturated} controller~$\gus$ by
\begin{align*}
  \gus(\zeta_j, &\uv_{j-1}) \triangleq
  \\
  &
  \begin{cases}
    \uv_{j-1}, & \text{if } \vv_j = 0 , \\
    \left( \frac{ \vv_{j-1} }{ \vv_j } \left( 1 + \sigma_j \frac{
          \uv_{j-1} }{ -u_m } \right) -1 \right) \left( \frac{ -u_m }{
        \sigma_j } \right), & \text{if } \vv_j > 0 .
  \end{cases}  
\end{align*}
The rationale behind this definition is as follows. As mentioned
above, it is sufficient to design a controller that ensures safe
following when $\vv_j(t) \geq \vv_{j-1}(t)$. Thus, if $\vv_j = 0$ then
we need to consider only the case of $\vv_{j-1} = 0$. In this case,
the definition of $\gus$ ensures that the vehicle $j$ stays at rest as
long as vehicle $j-1$ is at rest and starts moving only when $j-1$
starts moving again. Further, since the relative velocity and
acceleration in this case would be zero, we see that $\sigma_j$ stays
constant. As we see more thoroughly in the sequel, if the vehicle is
moving, $\vv_j > 0$, then $\guc$ ensures that $\sigma_j$ stays
constant and thus ensuring safety. However, in the second case, $\gus$
might cause $\vv_j$ to exceed $v^M$. Further, we would like the
vehicle to continue using the optimal uncoupled controller if it does
not affect the safety by decreasing $\sigma_j$.  These considerations
motivate our definition of the \emph{safe-following} controller as
\begin{multline}\label{eq:sf-control}
  \gsf(t, \zeta_j, \uv_{j-1}) \triangleq
  \\
  \min \lbrace \guc(\tau_j, t, \xv_j, \vv_j), \gus(\zeta_j, \uv_{j-1})
  \rbrace .
\end{multline}



\subsubsection{\localvehcontrol controller}

Here, we design the local vehicle controller by specifying a rule to
switch between the uncoupled controller $\guc$ and the safe-following
controller~$\gsf$. To make precise whether two vehicles are
sufficiently far from each other, we introduce the \emph{coupling set}
$\Cc_s$ defined by
\begin{equation}\label{eq:coupling-set}
  \Cc_s \triangleq \{ ( \mathrm{v}_1, \mathrm{v}_2, \sigma ) :
  \mathrm{v}_2 \geq \mathrm{v}_1 \text{ and } \sigma \in [ 1, \sigma_0
  ] \} ,
\end{equation}
with $\sigma_0 > 1$ a design parameter. Intuitively, if $\zeta_j \in
\Cc_s$, then vehicle $j$ is going at least as fast as the vehicle in
front of it, and their safety ratio is close to $1$. With this in
mind, we define the \localvehcontrol controller for vehicle $j$,
\begin{align}\label{eq:vehicle-control}
  \uv_j(t) =
  \begin{cases}
    \guc, &\text{if } \zeta_j \notin \Cc_s, \ \vv_j < v^M ,
    \\
    [ \guc ]_{u_m}^0, &\text{if } \zeta_j \notin \Cc_s, \ \vv_j = v^M ,
    \\
    \gsf, &\text{if } \zeta_j \in \Cc_s, \ \vv_j < v^M ,
    \\
    [ \gsf ]_{u_m}^0, &\text{if } \zeta_j \in \Cc_s, \ \vv_j = v^M .
  \end{cases}
\end{align}
Note that $[ \guc ]_{u_m}^0 \neq \guc$ only if $\nexists \Fc_j$. This
controller has the vehicle use the safe-following controller when in
the coupling set, and the uncoupled controller otherwise.

\subsection{Upper bound on guaranteed occupancy
  time}\label{sec:bartauocc}

The last element of the design is the upper bound on the guaranteed
occupancy time for a bubble. To obtain this, we first upper bound the
inter-approach times of vehicles in a given bubble at the beginning of
the intersection. 

\begin{proposition}\longthmtitle{Upper bound on the inter-approach
    times of vehicles in a bubble at the
    intersection~\cite{PT-JC:17-tcns}} \label{prop:int-app-times}
  For any bubble $i$ and any vehicle $j \in \{ 2, \ldots, m_i \}$, if
  $\vv_{i,j-1}(\Ta_{i,j-1}) \geq \nunom$, then $\vv_{i,j}(\Ta_{i,j})
  \geq \nunom$ and $\Ta_{i,j} - \Ta_{i,j-1}$ is upper bounded by
  \begin{equation*}
    \Tiat \triangleq
    \begin{cases}
      \sigma_0 \Tnom, &\text{if } \ulinevv \geq \nunom ,
      \\
      \max\{ \sigma_0 \Tnom, \Tfol(\ulinevv) \}, &\text{if } \ulinevv
      < \nunom ,
    \end{cases}
  \end{equation*}
  where $\ulinevv \triangleq \frac{ -u_m v^M }{ -u_m + \sigma_0 u_M }$
  and
  \begin{align*}
    \Tfol(v) & \triangleq \frac{ (\nunom)^2 - v^2 }{ 2 u_M v^M } +
    \frac{ \sigma_0 \Dc( v, v^M ) }{ v^M } + \frac{ \nunom - v }{ u_M
    } .
  \end{align*}
\end{proposition}

As a consequence of Proposition~\ref{prop:int-app-times}, we can
guarantee an upper bound on the intersection occupancy time of a
bubble.

\begin{corollary}\longthmtitle{Guaranteed upper bound on occupancy
    time of a bubble~\cite{PT-JC:17-tcns}}\label{cor:occ-time}
  For any bubble~$i$, its occupancy time~$\tauocc_i$ is upper bounded
  as $\tauocc_i \le \bartauocc_i$, where
  \begin{equation}\label{eq:bartauocc}
    \bartauocc_i = (m_i - 1) \Tiat + \max \left\{
      \frac{ L + \Delta }{ \nunom } , \Tiat \right\} .
  \end{equation}
\end{corollary}

The reasoning for the inclusion of $\Tiat$ in the second term
of~\eqref{eq:bartauocc} is as follows. If the bubble $i'$ that uses
the intersection immediately after bubble $i$ is from the same branch
as $i$, then we would like to have a safe-following distance between
the last vehicle $(i,m_i)$ of bubble $i$ and the first vehicle
$(i',1)$ of bubble $i'$ at the time the vehicle $(i',1)$ approaches
the intersection at its assigned time, $\Ta_{i',1} = \tau_{i',1} =
\tau_{i'}$.

\section{Provably safe optimized traffic
  coordination}\label{sec:all-together}

This section brings together the discussion above on the individual
aspects (dynamic vehicle clustering into bubbles, optimized planning
and scheduling of the bubbles, and local distributed control for
safety and execution of plans) of our hierarchical-distributed
coordination approach to intersection traffic.  The following result
shows that the design ensures vehicle safety and satisfies the
prescribed schedule.

\begin{theorem}\longthmtitle{Provably safe optimized traffic
    coordination}\label{thm:sf-control}
  Consider a traffic intersection with four incoming branches
  operating under Assumptions (i)-(v) in
  Section~\ref{sec:problem-statement}, where the vehicle dynamics are
  given by~\eqref{eq:vehicle-dyn} under the \localvehcontrol
  controller~\eqref{eq:vehicle-control}.
  Assume the exit zone length satisfies $L_e \geq -{ (v^M)^2 }/{ 2u_m
  } + { (\nunom)^2 }/{ 2u_M }$ and that, at initial time $t_0 = 0$,
  vehicles on each branch $k \in \until{4}$ are within the staging
  zone.  Furthermore, suppose that at each $t_s = s \timestepCS$ for
  each $s \in \integersnonnegative$, the vehicles in the staging zone
  that are clustered by the \clusteralgo algorithm are in a safe
  configuration ($\sigma_j(t_s) \geq 1$ for each new vehicle $j$).
  Then,
  \begin{enumerate}
  \item each vehicle belongs to some cluster, each bubble is scheduled
    by the \schedalgo algorithm at least once. Moreover, at each
    $t_s$, this strategy optimizes the schedule of the bubbles $\Lc$
    given by the \clusteralgo algorithm by minimizing the simplified
    cost function $\Cc_\Lc$,
  \item the schedule assigned to the bubbles respects the
    non-collision constraints~\eqref{eq:tau-constraints}, with the
    occupancy time of each bubble $i$ upper bounded by $\bartauocc_i$
    as given in~\eqref{eq:bartauocc},
  \item inter-vehicle safety is ensured ($\sigma_j \geq 1$) for all
    vehicles and for all time subsequent to $t_0$, and
  \item the first vehicle $(i,1)$ of each bubble $i$ approaches the
    intersection at $\tau_i$, the bubble uses the intersection only
    within its allotted time interval $[\tau_i, \tau_i +
    \bartauocc_i]$, and each vehicle travels with a velocity of at
    least $\nunom$ after approaching the intersection.
  \end{enumerate}
\end{theorem}
\begin{proof}
  (i) This claim follows from the \clusteralgo and the \schedalgo
  algorithms. Claim (ii) is ensured by the inclusion of the
  non-collision constraints~\eqref{eq:tau-constraints} in the
  \schedalgo algorithm and the feasibility of the scheduling problem
  guaranteed by Lemma~\ref{lem:zone-length-assump}.

  (iii) Inter-vehicular safety is a consequence
  of~\cite[Lemma~IV.1]{PT-JC:17-tcns}
  - for $\sigma_j \in [1, \sigma_0]$, if
  $\zeta_j \in \Cc_s$, then $\sigma_j$ either stays constant or
  increases; if on the other hand $\zeta_j \notin \Cc_s$, then it means
  $\vv_j < \vv_{j-1}$ and $\xv_{j-1} - \xv_j$ increases while
  $\Dc(\vv_{j-1},\vv_j)$ stays constant at $L$ and thus $\sigma_j$
  increases. Thus $\sigma_j(t) \geq 1$ is guaranteed for all vehicles
  $j$ and for all $t \geq t_s$.

  (iv) If no bubble precedes bubble $i$ on its branch, then the
  vehicle $(i,1)$ approaches the intersection at its designated time
  $\tau_{i,1} = \tau_i$, with at least a velocity of $\nunom$. Then,
  by applying Proposition~\ref{prop:int-app-times} inductively, we see
  that the last vehicle $(i,m_i)$ of bubble $i$ approaches the
  intersection with a velocity of at least $\nunom$ and $\Ta_{i,m_i}
  \leq \Ta_{i,1} + (m_i - 1) \Tiat$ and it takes at most
  $(L+\Delta)/\nunom$ amount of time to go past the intersection. Thus
  from~\eqref{eq:bartauocc}, we see that claim (iv) is satisfied in
  this case.

  Now suppose bubble $q$ precedes bubble $i$ on its branch and suppose
  claim (iv) is true for bubble $q$. From our reasoning above,
  $\Ta_{q,m_q} \leq \Ta_{q,1} + (m_q - 1) \Tiat = \tau_q + (m_q - 1)
  \Tiat$. Now, using arguments analogous to those in
  Proposition~\ref{prop:int-app-times}, we see that if $\nexists
  \Fc_{i,1}$ at any time $t \in [t_s, \Ta_{i,1}]$, then we have
  $\Ta_{i,1} \leq \Ta_{q,m_q} + \Tiat$. However, note
  from~\eqref{eq:tau-constraints} and~\eqref{eq:bartauocc} that
  \begin{equation*}
    \tau_i \geq \tau_q + \bartauocc_q \geq \tau_q + (m_q - 1) \Tiat +
    \Tiat \geq \Ta_{q,m_q} + \Tiat ,
  \end{equation*}
  where in obtaining the second inequality we have
  used~\eqref{eq:bartauocc}. Recall from the first paragraph of the
  proof of Proposition~\ref{prop:int-app-times} that for any vehicle
  $(i,j)$, $\Ta_{i,j} \geq \tau_{i,j}$ and in particular $\Ta_{i,1}
  \geq \tau_i$. Thus, we conclude that $\exists \Fc_{i,1}$ for all
  time $t \in [t_s, \Ta_{i,1}] = [t_s, \tau_i]$ and that vehicle
  $(i,1)$ approaches the intersection at its assigned time $\tau_i$
  with a velocity of at least $\nunom$. Hence, by using induction over
  all vehicles in bubble $i$ and over all bubbles $i$ themselves we
  conclude that claim (ii) holds.
\end{proof}

Theorem~\ref{thm:sf-control} does not guarantee the optimal operation
of the system at the level of individual vehicles under the proposed
hierarchical-distributed coordination approach. However, this result
guarantees the optimality at the level of bubbles, on each time $t_s =
s \timestepCS$, for the bubbles scheduled at~$t_s$. We believe this is
a good compromise in balancing the trade-off between optimal vehicle
operation and complexity of planning and control.

\section{Simulations}

This section presents simulations of our proposed
hierarchical-distributed design and comparisons with a signal-based
traffic coordination approach under varying traffic conditions. 
Table~\ref{tab:sys-par} specifies the system parameters that we keep
fixed across all the simulations presented here.
\begin{table}[!htb]
  \caption{System parameters} \label{tab:sys-par}
  \centering
  \begin{tabular}{l l l}
    Parameter & Symbol & Value \\
    \hline
    \multicolumn{3}{c}{General parameters} \\
    \hline
    Car length & $L$ & $4$m \\
    Intersection length & $\Delta$ & $12$m \\
    Zone lengths & $L_s, L_m, L_e$ & $70$m \\
    Max. speed limit & $v^M$ & $60$km/h \\
    Max. accel. & $u_M$ & $3$m/s$^2$ \\
    Min. accel. & $u_m$ &  $-4$m/s$^2$ \\
    \hline
    \multicolumn{3}{c}{HD algorithm parameters} \\
    \hline
    Nominal speed of crossing & $\nunom$ & $48$km/h \\
    Parameter in~\eqref{eq:coupling-set} & $\sigma_0$ & 1.2 \\
    Nominal inter-vehicle approach time & $\Tnom$ & $\approx 1.23$s
    \\
    Upper-bound on inter-vehicle approach time & $\Tiat$ &
    $\approx1.58$s \\
    Time period for execution of Algorithm~\ref{algo:dyn-CS} &
    $\timestepCS$ & $3.77$s \\
    Max. \# new bubbles on branch $k$ & $\Ncbar_k$ & 2 \\
    Max. \# bubbles scheduled & $\maxNbubbles$ & 8 \\    
    \hline
    \multicolumn{3}{c}{Signal-based algorithm parameters} \\
    \hline
    Green-time & & 10s
  \end{tabular}
\end{table}
The parameters $\Tnom$ and $\Tiat$ are computed parameters, while the
remaining ones in the table are design choices.

\subsection{Dynamic traffic generation}

In order to simulate dynamically generated traffic, new vehicles are
spawned every $\timestepCS$ units of time anywhere in the staging zone
of each branch as a Poisson arrival process~\cite{AP-SUP:02}. In order
to ensure safe following at the moment vehicle $j$ is spawned, we
define the arrival process on $\sigma_j$ rather than on vehicle
distances. To be precise, let $t_s = s \timestepCS$ be the time at
which vehicle $j$ is spawned. Then, we let $\sigma_j(t_s) = \sigma$,
with $\sigma$ drawn from the distribution with probability density
function $1 + (1/\mu) \text{e}^{-(1/\mu) \sigma}$ so that
$\sigma_j(t_s) \geq 1$ and the distribution has a mean of $1 +
\mu$. Thus, the smaller the value of $\mu$ the greater is the density
of the generated traffic.

The velocity of vehicle $j$ at $t_s$, $\vv_j(t_s)$, takes a random
value uniformly chosen from the interval $[0, v^M]$. Then, the
position of vehicle $j$ at $t_s$, $\xv_j(t_s)$ is obtained
from~\eqref{eq:sigma} as
\begin{equation*}
  \min \{ -(L_e+L_m), \xv_{j-1}(t_s) - \sigma_j(t_s)
  \Dc(\vv_{j-1}(t_s), \vv_j(t_s)) \},
\end{equation*}
where $\xv_{j-1}(t_s)$ and $\vv_{j-1}(t_s)$ are the position and the
velocity of the last vehicle previously defined on the same branch as
that of vehicle $j$. If there is no previously defined vehicle on the
branch, then $\xv_{j-1}(t_s) = \infty$. Recall from
Section~\ref{sec:cluster-dyn} that if $\timestepCS < \frac{ L_s }{ v^M
}$, then vehicles entering the problem domain during the time interval
$[t_s - \timestepCS, t_s]$ are within the staging zone at~$t_s$. This
explains the imposition of the upper bound $-(L_e+L_m)$ on
$\xv_j(t_s)$. Finally, the number of vehicles spawned on a branch at
$t_s$ is determined as follows. The procedure to spawn new vehicles
described above is repeated as long as the spawned vehicle is in the
staging zone, i.e., $\xv_j(t_s) \in [-(L_e+L_m+L-s), -(L_e+L_m)]$.

\subsection{Signal-based traffic coordination}

The vehicle control policy in the simulations with signal-based
traffic coordination is given by~\eqref{eq:vehicle-control} with
$\gsf$ given by~\eqref{eq:sf-control} and with $\guc = u_M$. The
traffic signaling policy is as follows - at any given time only one of
the four branches has the right of way (green or yellow signal). The
other branches do not have a right of way (red light). When the signal
changes to green from red for a branch, it stays green for a period of
time we call \emph{green-time} and then it turns to yellow. When the
signal for a branch is yellow the vehicles on that branch, which are
yet to cross the intersection, are divided into two groups - those
that can come to a stop before the beginning of the intersection and
those that cannot. For the first group of vehicles, a virtual vehicle
is introduced at position $L$ and having a velocity of $0$ so that the
control policy~\eqref{eq:vehicle-control} with the added virtual
vehicle ensures that the second group of vehicles comes to a stop
before the intersection. The second group of vehicles continue with
the usual control policy. The signal changes from yellow to red when
the first group of vehicles all cross the intersection. The virtual
vehicle for the branch is retained as long as the branch has yellow or
red signals. The branches get the right of way in a round-robin
manner.

\subsection{Results and discussion}

In the first set of simulations we simulated the signal-based
algorithm and the proposed hierarchical-distributed (HD) algorithm for
different values of $\mu$ and with $W_T = 1$ in the cost
function~\eqref{eq:cost-fun}. For each $\mu$, 10 trials were
conducted, each for 1 minute of simulation time. The results are
summarized in Figure~\ref{fig:sim-time-compare}. Note that as
per~\eqref{eq:bartauocc} the inter-approach times of vehicles is lower
bounded by $\Tiat$ in our simulations, which sets a uniform upper
bound on the intersection throughput of around 38 cars per minute
(CPM). This is the limiting factor, which explains the nearly constant
throughput across $\mu$ and across different trials in
Figure~\ref{fig:HD-CPM}. As can be seen from Figure~\ref{fig:sig-CPM},
the throughput in the signal-based coordination is significantly
higher, at the expense of higher cost.
\begin{figure*}[!htpb]
  \centering \subfigure[\label{fig:sig-CPM}]
  {\includegraphics[width=0.24\textwidth]{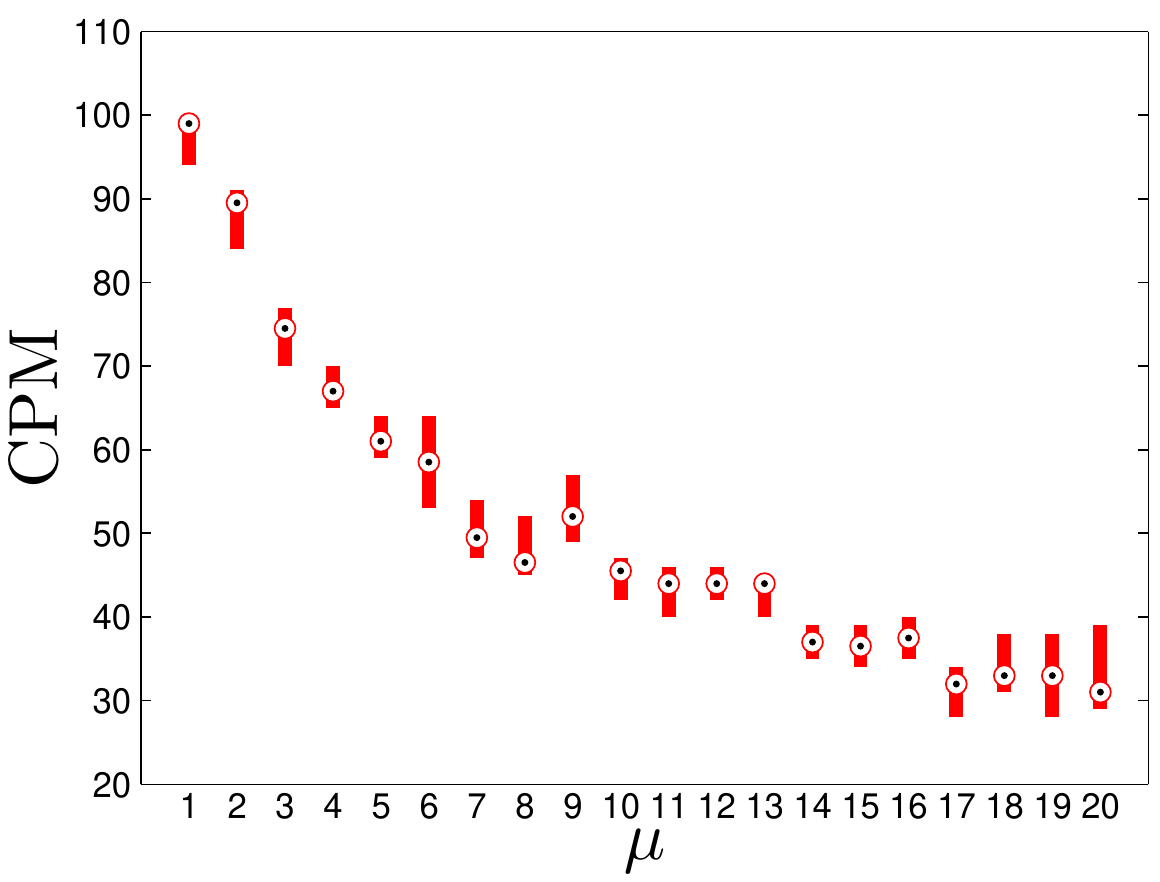}}
  \subfigure[\label{fig:HD-CPM}]
  {\includegraphics[width=0.24\textwidth]{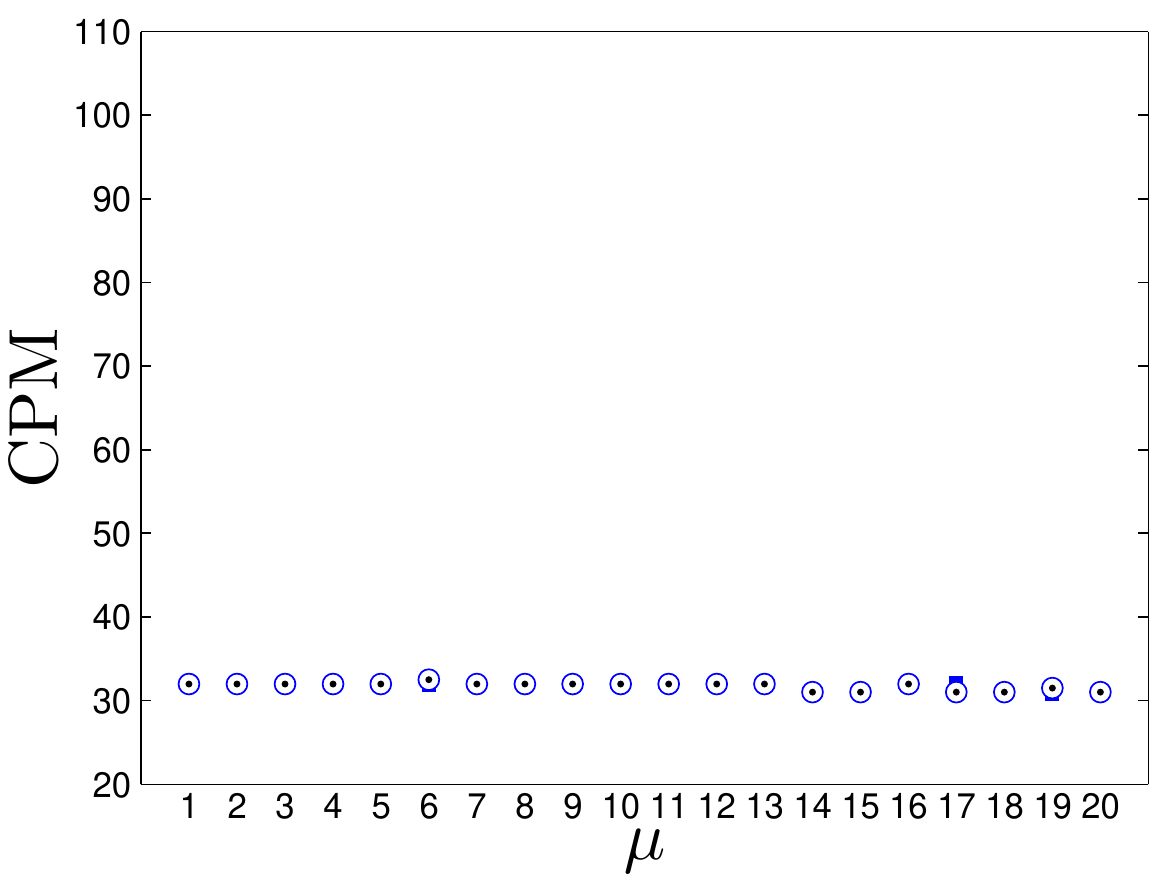}}
  \subfigure[\label{fig:sig-CPC-time}]
  {\includegraphics[width=0.24\textwidth]{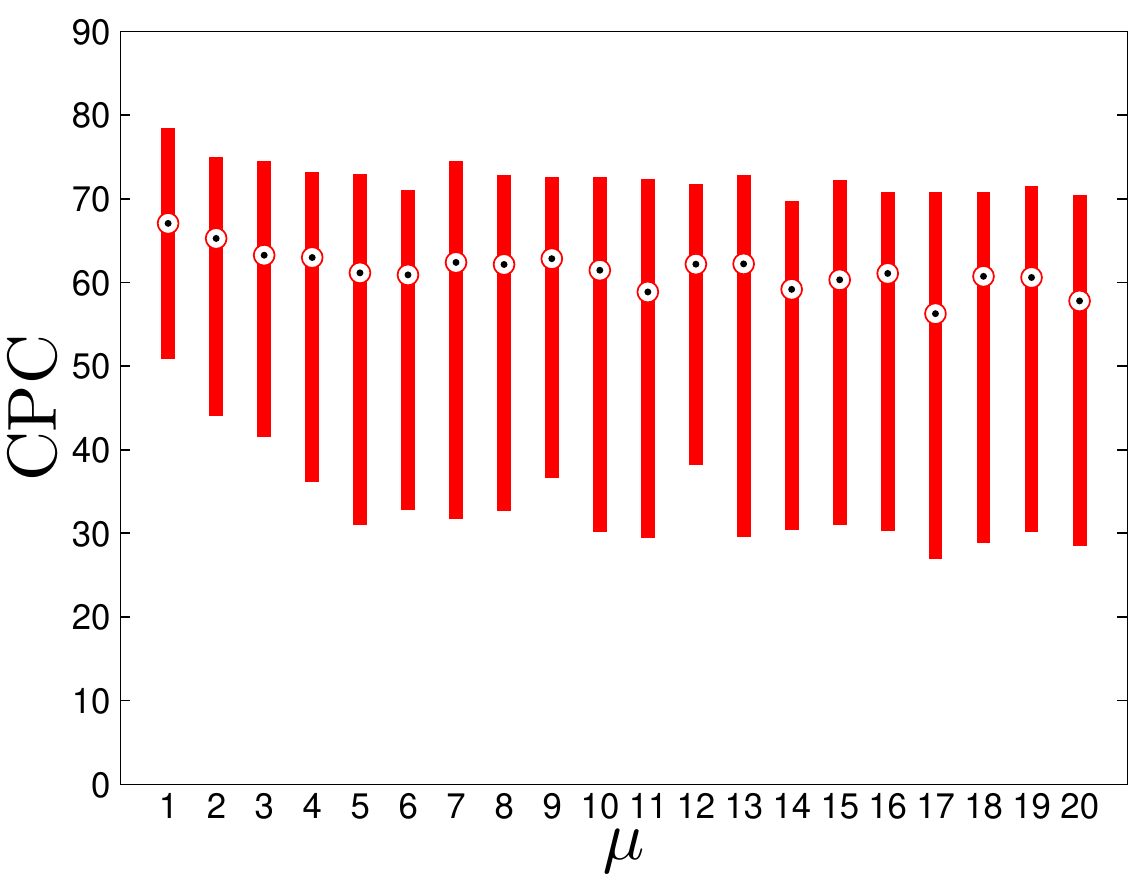}}
  \subfigure[\label{fig:HD-CPC-time}]
  {\includegraphics[width=0.24\textwidth]{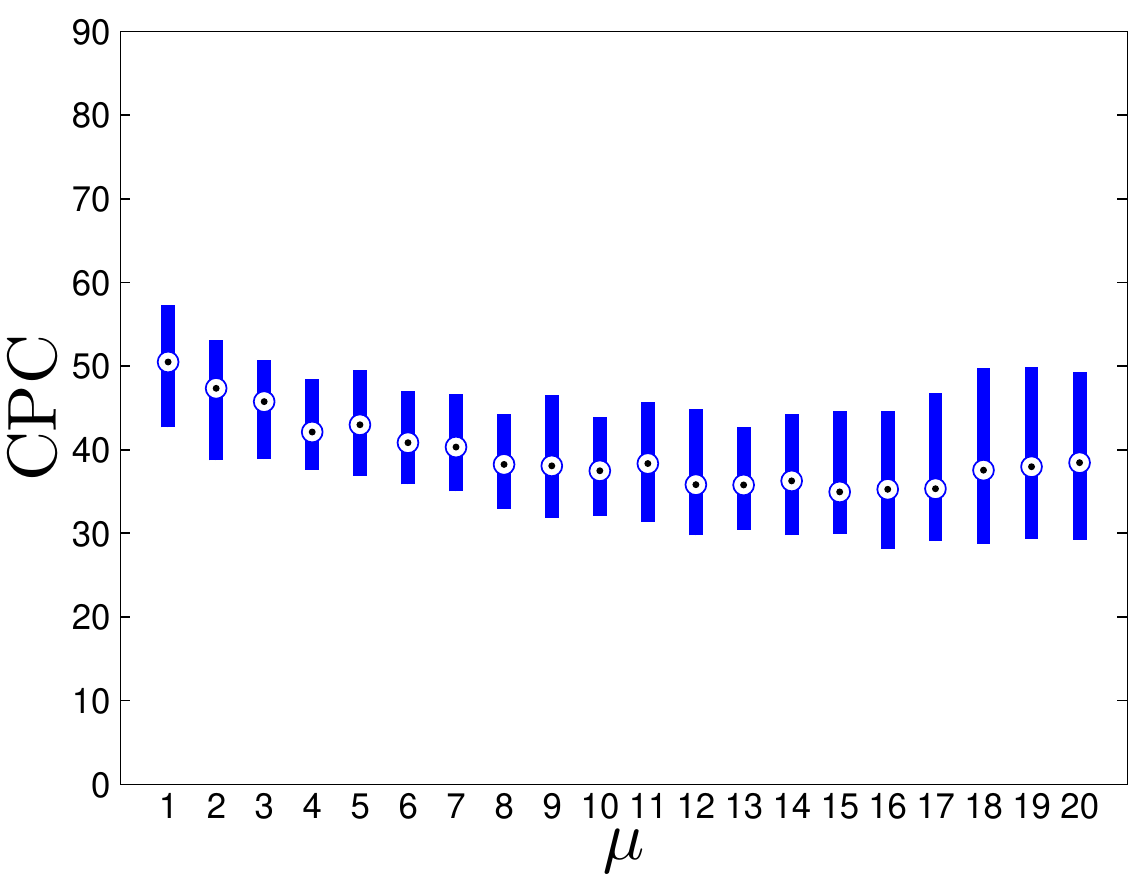}}
  \vspace*{-1ex}
  \caption{Cars per minute (CPM) is the number of cars that cross the
    intersection in a simulation time of $1$ minute. The cost per car
    (CPC) is computed with $W_T = 1$ in~\eqref{eq:cost-fun}. (a) and
    (c) are for the signal-based control while (b) and (d) plots are
    for the HD algorithm.} \label{fig:sim-time-compare}
\end{figure*}
For this reason, to have a fairer comparison, we have performed
simulations with the simulation time in each trial determined by the
time it takes 50 cars (TCC) to cross the intersection, summarized in
Figure~\ref{fig:sim-carcap-compare}. In a loose sense,
Figures~\ref{fig:sig-CPM} and~\ref{fig:sig-TCC} on the one hand and
Figures~\ref{fig:HD-CPM} and~\ref{fig:HD-TCC} on other hand are
inverted. Note that both in Figures~\ref{fig:sim-time-compare}
and~\ref{fig:sim-carcap-compare} the HD algorithm performs better in
terms of cost. Also note the lesser dispersion in the case of the HD
algorithm, which points to a more socially equitable distribution of
the cost. Also note that nearly steady throughput is a valuable
feature for traffic management for a network of intersections, such as
a city grid.
\begin{figure*}[!htpb]
  \centering \subfigure[\label{fig:sig-TCC}]
  {\includegraphics[width=0.24\textwidth]{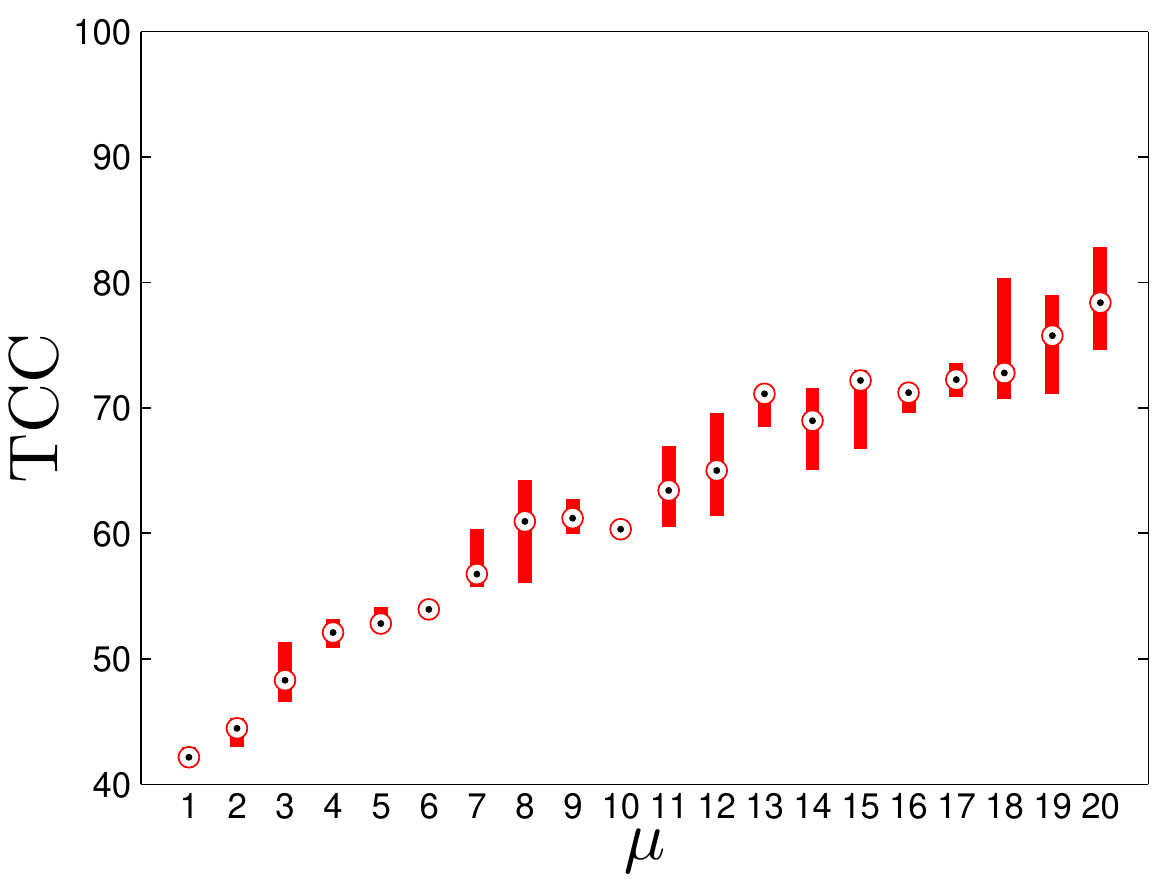}}
  \subfigure[\label{fig:HD-TCC}]
  {\includegraphics[width=0.24\textwidth]{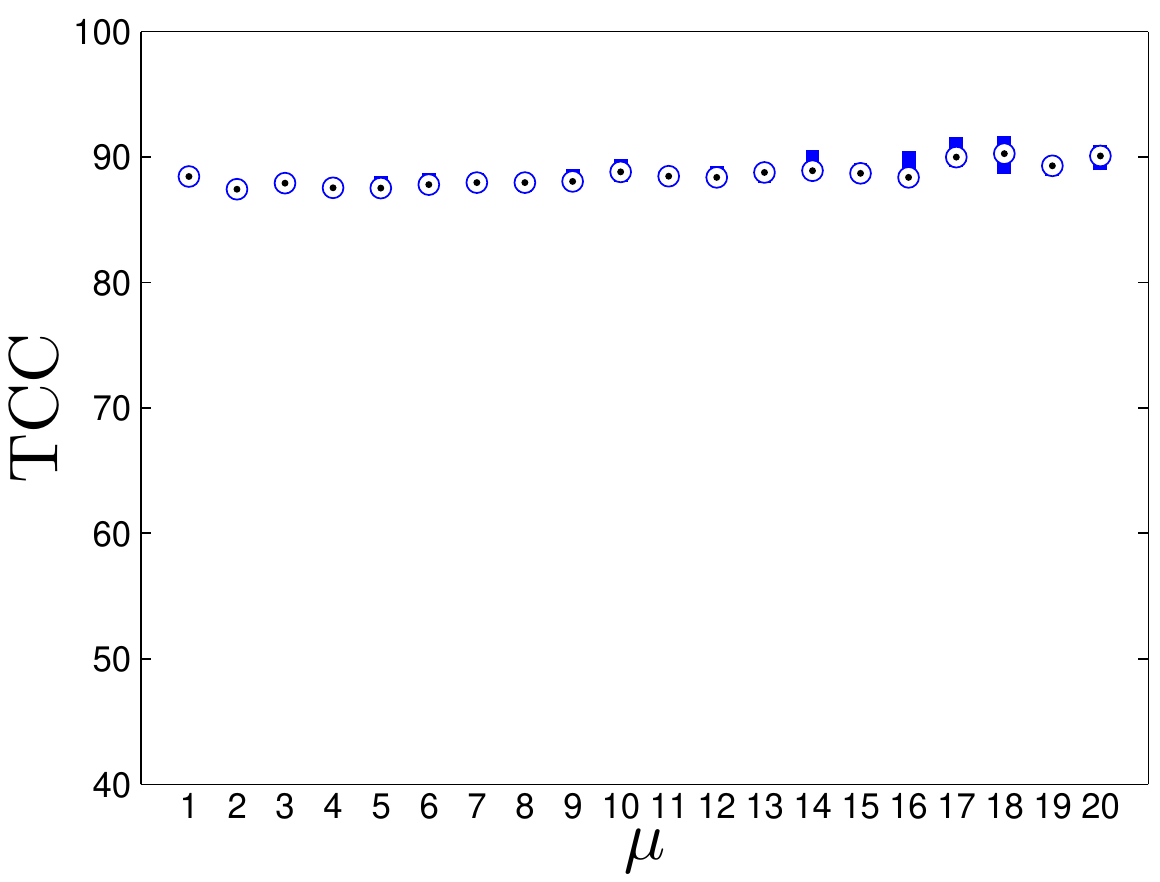}}
  \subfigure[\label{fig:sig-CPC-cap}]
  {\includegraphics[width=0.24\textwidth]{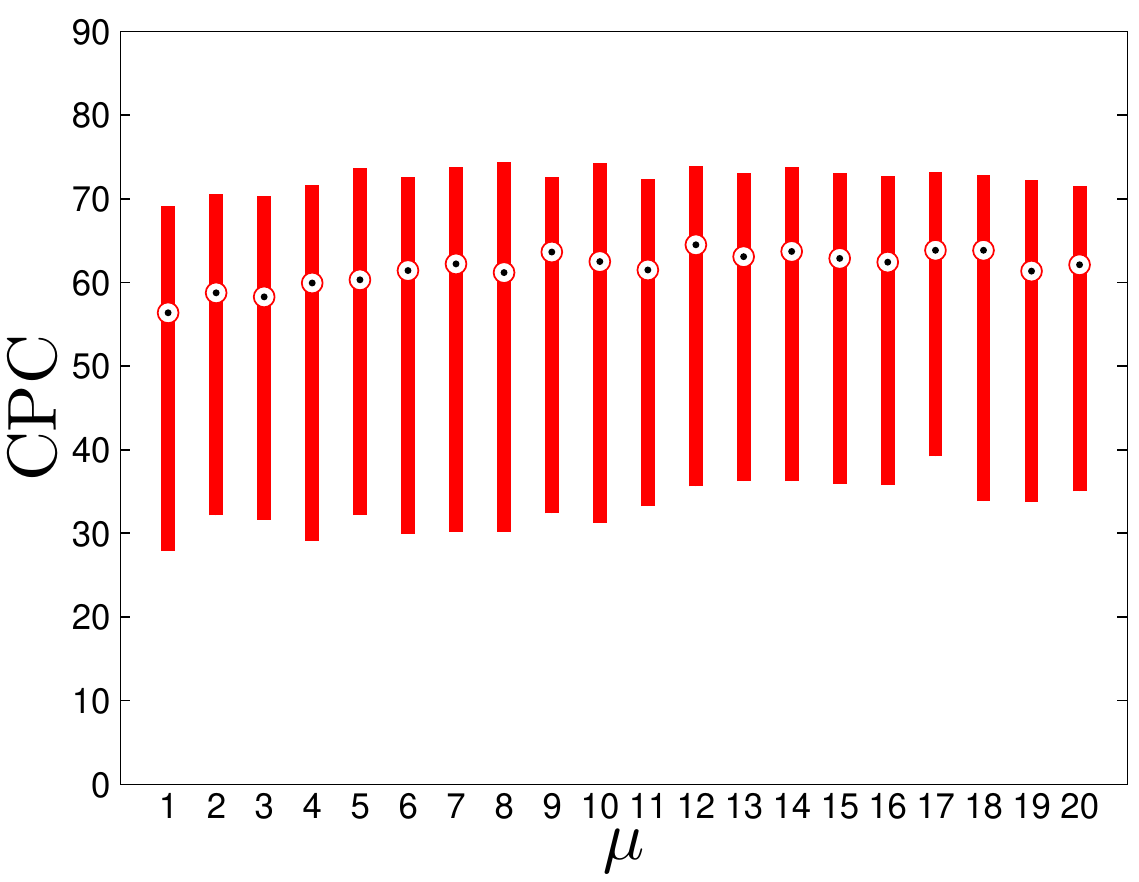}}
  \subfigure[\label{fig:HD-CPC-cap}]
  {\includegraphics[width=0.24\textwidth]{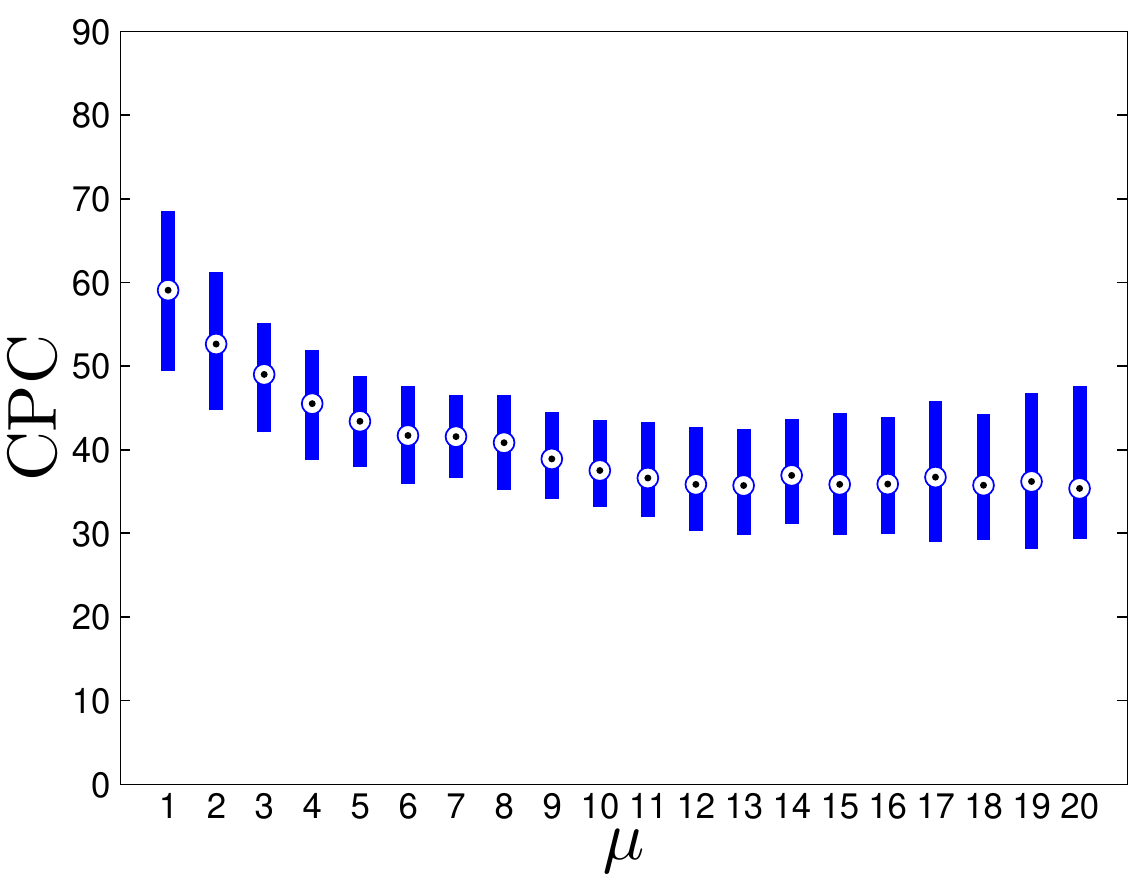}}
  \vspace*{-1ex}
  \caption{Time taken to reach car cap (TCC) is time taken for 50 cars
    that crossed the intersection.  The cost per car (CPC) is computed
    with $W_T = 1$ in~\eqref{eq:cost-fun}. (a) and (c) are for the
    signal-based control and (b) and (d) plots are for the HD
    algorithm.} \label{fig:sim-carcap-compare}
\end{figure*}

Figures~\ref{fig:Tcap-comp1} and~\ref{fig:Cost-comp1} summarize the
results of simulations performed for different values of $\mu$ and
$W_T$ in the cost function~\eqref{eq:cost-fun}, varied from $0.1$
to~$10$. The throughput is consistently better in the signal-based
control except for low-density traffic (high $\mu$). In terms of cost,
except in the cases with very high density traffic (low $\mu$) and
high weightage to travel time in the cost function, the HD algorithm
does better than the signal-based control. 

\emph{Computational expense:} Simulations presented here were
performed on an Intel Core i3-3227U processor in MATLAB R2014a running
on the Linux Mint 17.2 operating system. The most computationally
expensive component of our design is the branch-and-bound algorithm
for scheduling the bubbles. For scheduling 8 bubbles, this typically
took about 1 second. Next most expensive is the $k-$means algorithm,
which took about 13ms per instance. The controller $\guc$ was
implemented based on an explicit solution to the optimal control
problem. On average, an instance of $\guc$ was executed in about
$0.5$ms when the optimal control problem was feasible.



\begin{figure*}[!htpb]
  \centering
  \subfigure[\label{fig:HD-Tcap}]
  {\includegraphics[width=0.3\textwidth]{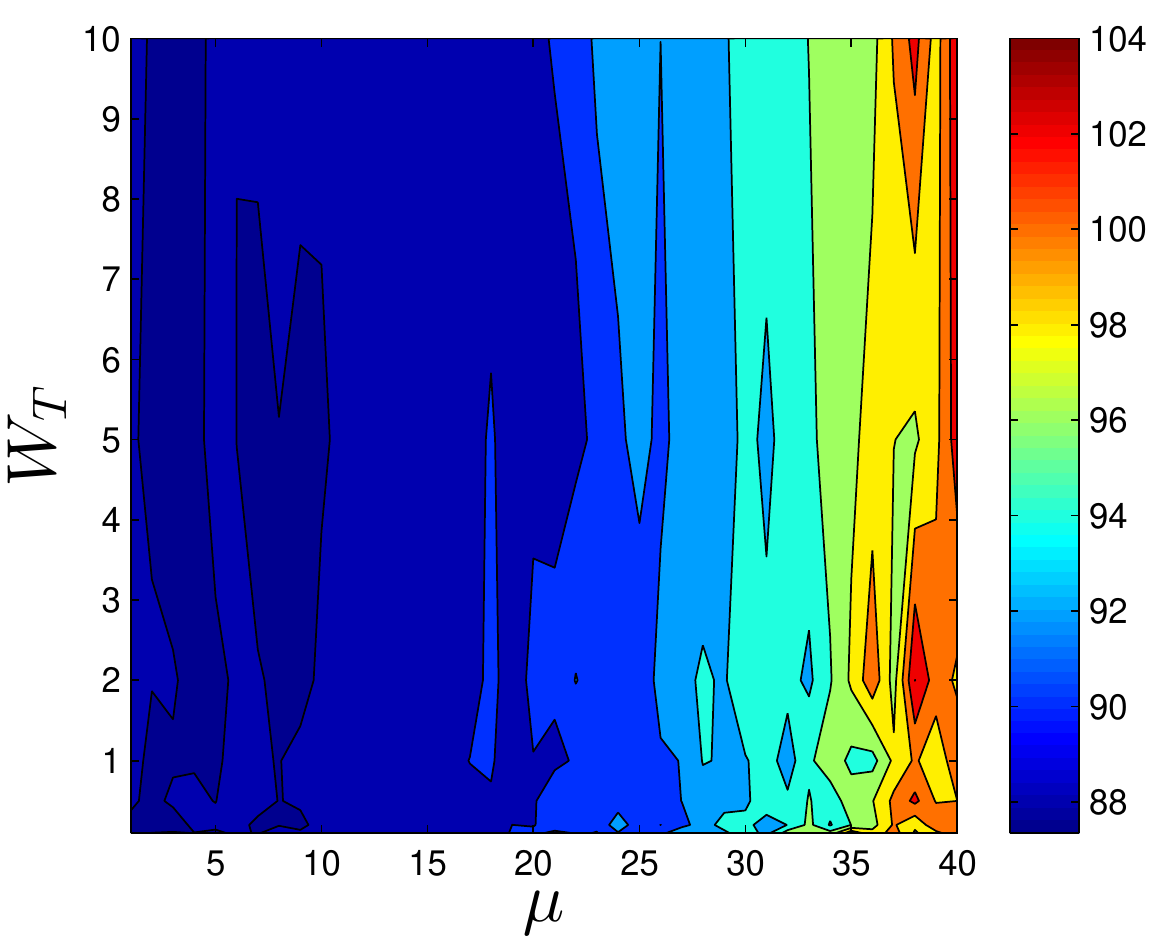}}
  \subfigure[\label{fig:Tcap-comp1}]
  {\includegraphics[width=0.3\textwidth]{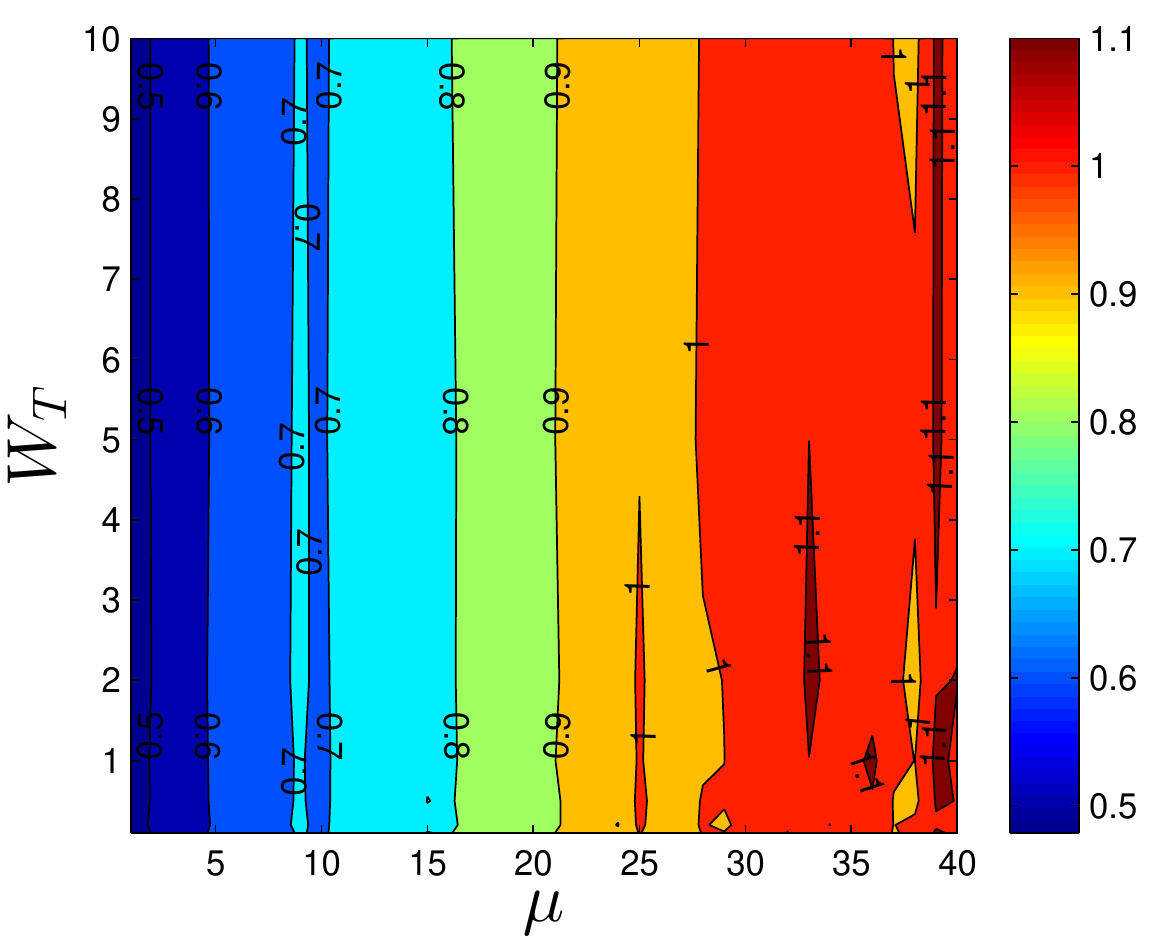}}
  \subfigure[\label{fig:Cost-comp1}]
  {\includegraphics[width=0.3\textwidth]{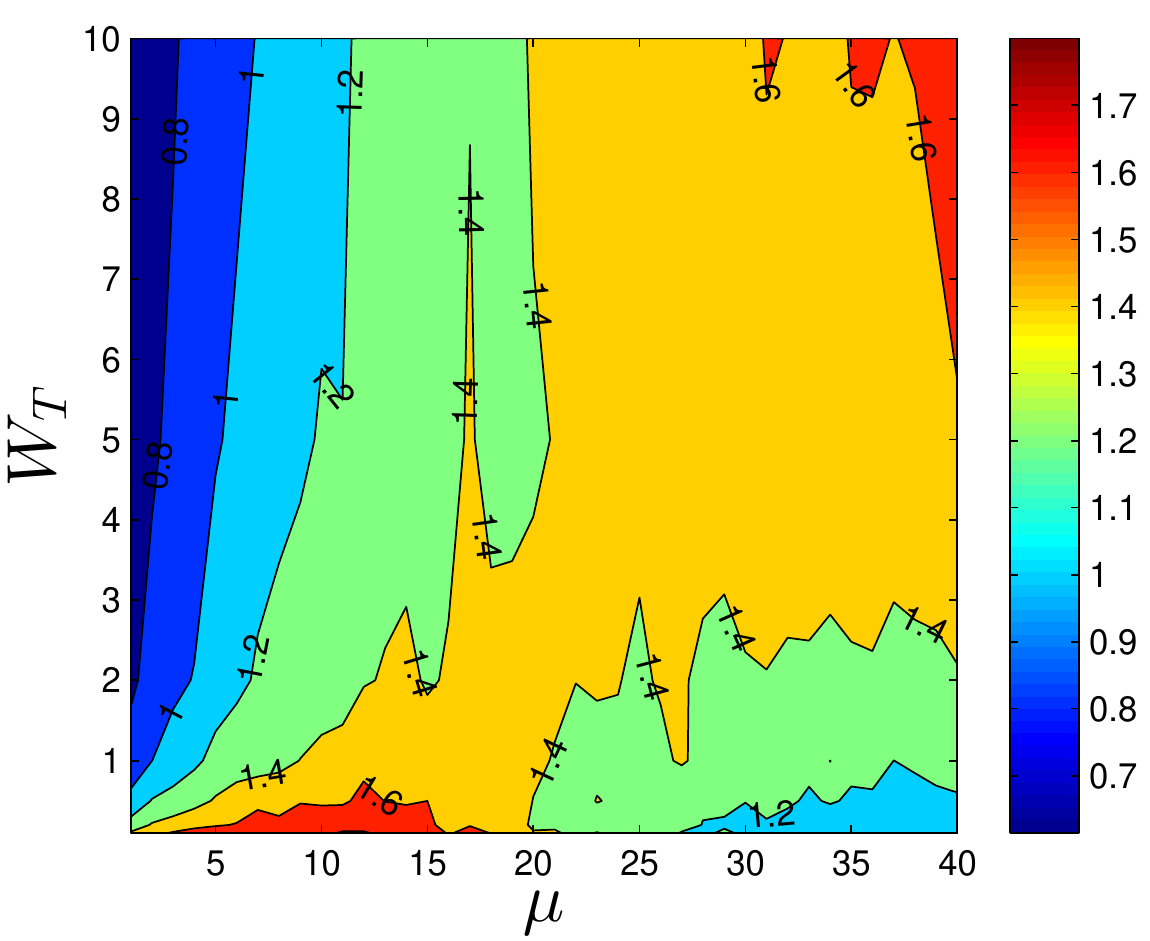}}
  \vspace*{-1ex}
  \caption{Summary of simulations with a cap of $50$ cars for various
    values of $\mu$ and the weight $W_T$ in the cost
    function~\eqref{eq:cost-fun}. (a) Average (over 10 trials) time
    taken for 50 cars to cross the intersection for the HD
    algorithm. The ratio of the average (over 10 trials) time (b) and
    cost (c) taken for 50 cars to cross the intersection for the
    signal-based coordination over the HD
    algorithm.}\label{fig:sim-bulk}
\end{figure*}

\section{Conclusions}

We have studied the problem of coordinating traffic at an intersection
in order to reduce travel time and improve vehicle energy efficiency
while avoiding collisions. Our provably correct intersection
management solution relies on communication among vehicles and the
infrastructure, and combines hierarchical and distributed control to
optimally schedule the passage of vehicle bubbles through the
intersection. Our dynamic bubble-based approach has the advantage of
reducing the complexity of the computationally intensive scheduling
problem and making the solution applicable for different traffic
conditions. Simultaneously, the modular nature of the major aspects of
our design make it easier to make improvements in the future. Finally,
since the central traffic manager at the intersection requires only
aggregate data of a bubble, this decomposition provides a certain
amount of privacy. We have performed simulations to illustrate the
performance of our design and compared it against a traditional,
signal-based intersection management approach.  Our
hierarchical-distributed algorithm performs better than signal-based
control in terms of cost except for high traffic densities and high
weightage to travel time in the cost function.  The guaranteed
throughput is, however, worse due to the conservativeness of the upper
bound on inter-approach times of the vehicles.  We believe further
analysis would improve this component and yield better throughput.
Other future work will study the computational complexity of the
proposed algorithm, the characterization of the expected size of the
generated bubbles, the incorporation of information about incoming
traffic density to improve throughput,
the inclusion of privacy preservation requirements, and the extension
to coordinated management for networks of intersections.

\section*{Acknowledgments}
The research was supported by NSF Award CNS-1446891.

\bibliographystyle{ieeetr} %
\bibliography{alias,FB,JC,Main,Main-add}

\begin{IEEEbiography}[{\includegraphics[width=1in,
    height=1.25in,clip,keepaspectratio]
    {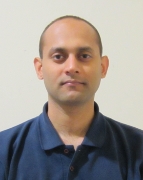}}]{Pavankumar Tallapragada}
  received the B.E. degree in Instrumentation Engineering from SGGS
  Institute of Engineering $\&$ Technology, Nanded, India in 2005,
  M.Sc. (Engg.) degree in Instrumentation from the Indian Institute of
  Science, Bangalore, India in 2007 and the Ph.D. degree in Mechanical
  Engineering from the University of Maryland, College Park in
  2013. He is currently a Postdoctoral Scholar in the Department of
  Mechanical and Aerospace Engineering at the University of
  California, San Diego. His research interests include
  event-triggered control, networked control systems, distributed
  control and transportation and traffic systems.
\end{IEEEbiography}

\begin{IEEEbiography}[{\includegraphics[width=1in,
  height=1.25in,clip,keepaspectratio]{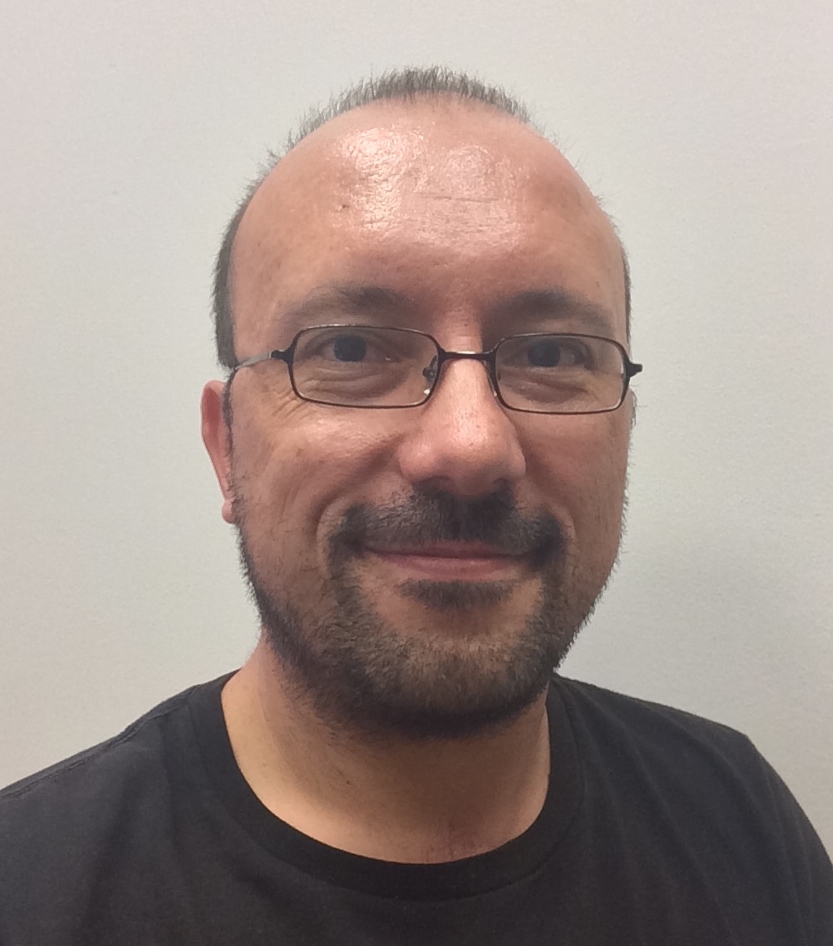}}]{Jorge
    Cort\'es}
  received the Licenciatura degree in mathematics from Universidad de
  Zaragoza, Zaragoza, Spain, in 1997, and the Ph.D. degree in
  engineering mathematics from Universidad Carlos III de Madrid,
  Madrid, Spain, in 2001.  He held post-doctoral positions with the
  University of Twente, Twente, The Netherlands, and the University of
  Illinois at Urbana-Champaign, Urbana, IL, USA. He was an Assistant
  Professor with the Department of Applied Mathematics and Statistics,
  University of California, Santa Cruz, CA, USA, from 2004 to 2007. He
  is currently a Professor in the Department of Mechanical and
  Aerospace Engineering, University of California, San Diego, CA,
  USA. He is the author of Geometric, Control and Numerical Aspects of
  Nonholonomic Systems (Springer-Verlag, 2002) and co-author (together
  with F. Bullo and S. Mart\'inez) of Distributed Control of Robotic
  Networks (Princeton University Press, 2009). He is an IEEE Fellow
  and an IEEE Control Systems Society Distinguished Lecturer.  His
  current research interests include distributed control, networked
  games, power networks, distributed optimization, spatial estimation,
  and geometric mechanics.
\end{IEEEbiography}

\end{document}